\documentclass{amsart}

\usepackage{tikz}
\usetikzlibrary{shapes,arrows,patterns}
\usepackage{enumerate}
\usepackage{float}
\usepackage{blkarray}
\usepackage[T1]{fontenc}

\usepackage{amsthm,amsmath,amssymb,mathtools}
\usepackage{graphicx}

\usepackage[colorlinks=true,citecolor=blue,linkcolor=red,urlcolor=blue]{hyperref}

\newcommand{\arxiv}[1]{\href{http://arxiv.org/abs/#1}{\texttt{arXiv:#1}}}
\newtheorem{theorem}{Theorem}
\newtheorem{corollary}[theorem]{Corollary}
\newtheorem{lemma}[theorem]{Lemma}
\theoremstyle{remark}
\newtheorem{remark}[theorem]{Remark}
\newtheorem{example}[theorem]{Example}

\begin{document}

\title[Graph reduction techniques and the multiplicity of the Laplacian eigenvalues]{Graph reduction techniques and\\ the multiplicity of the Laplacian eigenvalues}%

\author{Asghar Bahmani}
\address{Department of Mathematics and Computer Science, Amirkabir University of Technology, 424, Hafez Ave., Tehran 15914, Iran.}
\email{asghar.bahmani@aut.ac.ir}
\author{Dariush Kiani}
\address{Department of Mathematics and Computer Science, Amirkabir University of Technology, 424, Hafez Ave., Tehran 15914, Iran, and School of Mathematics, Institute for Research in Fundamental Sciences (IPM), P.O. Box 19395-5746, Tehran, Iran.}
\email{dkiani@aut.ac.ir, dkiani7@gmail.com}


\begin{abstract}
Let $M=[m_{ij}]$ be an $n\times m$ real matrix, $\rho$ be a nonzero real number, and $A$ be a symmetric real matrix. We denote by $D(M)$ the $n\times n$ diagonal matrix $diag(\sum_{j=1}^{m}m_{1j},\ldots,\sum_{j=1}^{m}m_{nj})$ and denote by $L_{A}^{\rho}$ the generalized Laplacian matrix $D(A)-\rho A$.
A well-known result of Grone et al. states that by connecting one of the end-vertices of $P_{3}$ to an arbitrary vertex of a graph, does not change the multiplicity of Laplacian eigenvalue $1$. 
We extend this theorem and some other results for a given generalized Laplacian eigenvalue $\mu$.
 Furthermore, we give two proofs for a conjecture by Saito and Woei on the relation between the multiplicity of some Laplacian eigenvalues and pendant paths.
\end{abstract}
\maketitle
\section{Introduction}

Let $A$ be a real symmetric matrix. There exists a unique weighted graph $G$ such that the adjacency matrix of $G$, is $A$; i.e. for $i\neq j$, $A_{ij}$ is the weight of the edge $\{i,j\}$ and $A_{ii}$ is  twice of the weight of the loop at the vertex $i$.  In this paper, we look at real symmetric matrices in this point of view. 

For a positive integer $n$, we denote by $\text{\rm Sym}_{n}(\mathbb{R})$, the set of real symmetric matrices of order $n$ and denote by $[n]$ the set $\{1,\ldots,n\}$.
The multiplicity of an eigenvalue $\lambda$ of $A$ is denoted by $m_{A}(\lambda)$ and a $\lambda$-eigenvector of $A$ is an eigenvector of $A$ corresponding to $\lambda$.
For two positive integers $i$ and $n$, $\pmb{j}_{n}$ and $\pmb{e}_{i}$, denote the all $1$'s vector and the vector with a $1$ in the $i^{\rm th}$ coordinate and $0$'s elsewhere, respectively, in $\mathbb{R}^{n}$.
The restriction of a vector $\pmb{x}$ to any index set $I$ is denoted by $\pmb{x}_{| I}$ and we denote the entry of $\pmb{x}$ corresponding to an index $u$, by $\pmb{x}(u)$.
The identity matrix is denoted by $\mathbb{I}_{n}$  or briefly  $\mathbb{I}$.
The path, the cycle, and the star graph on $n$ vertices are denoted by $P_{n}$, $C_{n}$, and $S_{n}$, respectively.

Let $M=[m_{ij}]$ be an $n\times m$ real matrix. The transpose of $M$ is denoted by $M^{T}$ and we denote by $D(M)$, the $n\times n$ diagonal matrix $diag(\sum_{j=1}^{m}m_{1j},\ldots,\sum_{j=1}^{m}m_{nj})$. For $\rho\in \mathbb{R}-\{0\}$, we denote by $L_{A}^{\rho}$, the \textit{generalized Laplacian matrix} $D(A)-\rho A$. If $\rho=1$ ($\rho=-1$) and $A(G)$ is the adjacency matrix  of a given graph $G$, then we have the Laplacian matrix $L_{A(G)}^{1}=L(G)$ (signless Laplacian matrix $L_{A(G)}^{-1}=Q(G)$, respectively).

Let $\mu$ be a Laplacian eigenvalue of $G$. We consider the results about the relation between $m_{L(G)}(\mu)$ and $m_{L(H)}(\mu)$, for a particular subgraph $H$ of $G$. 
We recall some of these results:

\begin{theorem}{\rm \cite{GMS}}\label{delp3}
If $G'$ is a graph obtained from $G$ by connecting one of the end-vertices of $P_{3}$ to an arbitrary vertex of $G$, then we have $m_{L(G)}(1) =m_{L(G^{\prime})}(1)$.
\end{theorem}

\begin{theorem}{\rm \cite{Ne}}
Let  $G$  be  any graph  with a simple  Laplacian eigenvalue $\mu$.  Let  $u$ be 
a vertex of $G$  such that an eigenvector corresponding to $\mu$  is nonzero on $u$. 
Let $H$ be  any graph,  and  let $G^{\prime}$ be the graph formed  by joining  an arbitrary 
vertex of  $H$ to $u$.  Then $m_{L(H)}(\mu)=m_{L(G^{\prime})}(\mu)$.
\end{theorem}

A connected sum of two graphs $G_{1}$ and $G_{2}$ is any graph $G$ where $V(G)=V(G_{1})\cup V(G_{2})$ and $E(G)$
differs from $E(G_{1})\cup E(G_{2})$ by the addition of a single edge joining some (arbitrary) vertex of $V(G_{1})$
to some vertex of $V(G_{2})$, and is denoted by $G=G_{1}\#G_{2}$ \cite{GMS}.
\begin{theorem}{\rm \cite{GMS}}\label{delsn}
Let $G$ be a nonempty graph on $n$ vertices. Let $H = G\#S_{k}$ be a connected sum of $G$ with the star on $k>1$ vertices. Then $m_{L(G)}(k) =m_{L(H)}(k)$
\end{theorem}

A \textit{cluster} of a graph $G$ is an independent set of two or more vertices of $G$, each of which has the same set of neighbours. The degree of a cluster is the cardinality of its shared set of neighbours, i.e., the common degree of each vertex in the cluster. A $d$-cluster is a cluster of degree $d$. The number of vertices in a $d$-cluster is its order. A collection of two or more $d$-clusters is independent if the sets of vertices comprising the $d$-clusters are pairwise disjoint \cite{GM}.

\begin{theorem}{\rm \cite{GM}}\label{dcluster}
 Let $G$ be a graph with $k$ independent $d$-clusters of orders $r_{1},\ldots,r_{k}$. Then $m_{L(G)}(d)\geq \sum_{i=1}^{k} r_{i} -k$.
 \end{theorem}
 
 Among other results, we generalize these results for weighted graphs and an arbitrary generalized Laplacian eigenvalue $\mu$.

A \textit{pendant path} of a graph $G$ is a path such that one of its end vertices has degree one and all the internal vertices have degree two and other end vertex has degree greater than two.
$p_k(G)$ denotes the number of pendant paths of length $k$, and $q_k(G)$ n the number of vertices with degree greater than three which are an end vertex of some pendant paths of length $k$.
If $k=1$, we have the well-known result of Faria \cite{Fa} that $m_{L(G)}(1)\geq p_{1}(G) - q_{1}(G)$.
Saito and  Woei \cite{SW} conjectured that  for any positive integer $k$, any graph $G$ has some Laplacian eigenvalue with multiplicity at least $p_k(G)-q_k(G)$ and proved it for $k=2$. The following generalization of the conjecture has been proved in \cite{Gh}. We give two proofs for this theorem in the next sections.

\begin{theorem}{\rm \cite{Gh}}\label{pkqk}
Let $G$ be a graph. Then $\displaystyle 4\cos^{2}(\frac{\pi i}{2k + 1})$ for any $k\geq 1$ and $i =1,\ldots,k$, is
both a Laplacian and a signless Laplacian eigenvalue of $G$ with multiplicity at least $p_{k}(G) - q_{k}(G)$.
\end{theorem}


Let $A\in \text{\rm Sym}_{n}(\mathbb{R})$ and $\lambda$ be an eigenvalue of $A$ of multiplicity $k$. A set $U\subseteq [n]$ is a \textit{star set} for $\lambda$ (or $\lambda$-star set) of  $A$ if $|U|=k$ and $\lambda$  is not an eigenvalue of the submatrix of $A$ obtained by removing rows and columns with index in $U$. It is known that for every eigenvalue $\lambda$ there exists a $\lambda$-star set \cite{CRS}.

We recall the following theorem about star sets that we use in the next sections.
\begin{theorem}\label{starcorr}{\rm \cite[Theorem 7.2.6]{CRS}}
Let $U$ be  a $\lambda$-star set of $A$.   If $m_{A}(\lambda)=k$, then there exists a basis of eigenvectors $\{\pmb{\alpha}_{s}:s\in U\}$  such that $\pmb{\alpha}_{s}(t)=\delta_{st}$, whenever $s,t\in U$ and $\delta$ is the Kronecker delta function; its value is $1$ if $s=t$, and $0$ otherwise.
\end{theorem}

\section{Type I Reductions:  Edge Deleting}
In this section, for a given eigenvalue $\mu$, we remove a particular subgraph  corresponding to $\mu$ and consider the multiplicity of $\mu$ of remaining graph.

First, we state this following Edge  Principle Theorem.
\begin{theorem}{\rm \cite{Me}}
Let  $\mu$ be  a Laplacian  eigenvalue  of  $G$  afforded  by eigenvector  $\pmb{x}$. If $x_{i}=x_{j}$,  then  $\mu$ is  an  eigenvalue  of  $G'$ afforded  by  $\pmb{x}$,  where  $G'$  is the  graph  obtained  from  $G$ by  deleting  or  adding  $e=\{i, j\}$  depending  on  whether or  not  it  is  an  edge  of  $G$.
\end{theorem} 

Now, we state a weighted version of theorem above, for the Laplacian and the signless Laplacian of weighted graphs:

\begin{lemma}\label{edgeprin2}
Let  $n\in \mathbb{N}, \rho\in \{-1,1\}$, $A\in \text{\rm Sym}_{n}(\mathbb{R})$, and $\mu$ be  an  eigenvalue  of  $L_{A}^{\rho}$  with a $\mu$-eigenvector  $\pmb{x}$. Suppose that $a\in \mathbb{R}$ and  $x_{i}=\rho x_{j}$, for some $i,j\in [n],\, i\neq j$. If $A'$  is the  matrix  obtained  from  $A$ by setting  $A'_{ij}=A'_{ji}=A_{ij}+a$,  then $\pmb{x}$ is a $\mu$-eigenvector   of  $L_{A'}^{\rho}$.
\end{lemma} 
\begin{proof}
We have {\footnotesize ${L_{A}^{\rho}}=
{L_{A'}^{\rho}}-
\begin{blockarray}{ccccc}
\begin{block}{c(cccc)}
i &   & a      & -\rho a  &   \\
j &  &  -\rho a       & a  &  \\
\end{block}
\end{blockarray}$.} So,\\
{\footnotesize
$
\mu\pmb{x}={L_{A}^{\rho}}\pmb{x}=
{L_{A'}^{\rho}}\pmb{x}-
\begin{blockarray}{ccccc}
 & &i&j& \\
\begin{block}{c(cccc)}
 & &    &  &  \\
i &   & a      & -\rho a  &   \\
j &  &  -\rho a       & a  &  \\
  & &      &      &  \\
\end{block}
\end{blockarray}\,\pmb{x}=
{L_{A'}^{\rho}}\pmb{x}-
\begin{blockarray}{cc}
\begin{block}{c(c)}
 & 0 \\
i & a x_i-\rho a x_j      \\
j & - \rho a  x_i+ a x_j  \\
  & 0 \\
\end{block}
\end{blockarray}=
{L_{A'}^{\rho}}\pmb{x}$.}
\end{proof}

The following theorem is the main theorem of this section.

\begin{theorem}\label{symmetricII}
Let $\mu\in \mathbb{R}$, $\rho\in \{-1,1\}$, and $\mathcal{H},\mathcal{L}$ be real symmetric matrices with row and column indices $I=I_{1}\dot{\cup} I_{2}\dot{\cup} I_{3}$ and $J=J_{1}\dot{\cup} J_{2}$, respectively. Suppose that $I_{1}\dot{\cup} I_{2}$ is a $\mu$-star set of $L_{\mathcal{H}}^{\rho}$. If  $X,\mathcal{G},$ and $\mathcal{E}$ are matrices given below,
\begin{equation*}
\resizebox{ 0.7\textwidth}{!} 
{$
\newcommand*{\temp}{\BAmulticolumn{1}{|c}{}}
\newcommand*{\tempa}{\BAmulticolumn{1}{|c}{\mathcal{A}}}
\newcommand*{\tempj}{\BAmulticolumn{1}{|c}{J_{2}}}
X=
\begin{blockarray}{ccccc}
\begin{block}{c(cccc)}
  I_{11} & \pmb{x}_{1}  & \pmb{0}&\cdots &   \pmb{0}   \\
  I_{12} &  \pmb{0} &  \pmb{x}_{2}& \cdots &  \pmb{0} \\
  \vdots & \vdots &  \vdots &  \ddots &  \vdots   \\
  I_{1|J_{1}|} & \pmb{0}&  \pmb{0} & \cdots & \pmb{x}_{|J_{1}|}  \\
\end{block}
\end{blockarray}
\,,\,\,
\mathcal{G}=
\begin{blockarray}{cccc|cc}
 & I_{1} & I_{2} & I_{3}  & J_{1}  &\tempj \\
\begin{block}{c(ccc|cc)}
  I_{1} & & & &  X  &\temp  \\
  I_{2} & & \mathcal{H}& &  0 &\tempa\\
  I_{3} &  & & & 0 &\temp  \\\cline{1-6}
  J_{1} & X^{T} & 0 & 0 &    &  \\\cline{1-4}
  J_{2} &  &\mathcal{A}^{T} & & \mathcal{L}    & \\
\end{block}
\end{blockarray}\,,\,\,
\mathcal{E}=
\begin{blockarray}{cccc|cc}
 & I_{1} & I_{2} & I_{3}  & J_{1}  &\tempj \\
\begin{block}{c(ccc|cc)}
  I_{1} & & & &  0  &\temp  \\
  I_{2} & & \mathcal{H}& &  0 &\tempa\\
  I_{3} &  & & & 0 &\temp  \\\cline{1-6}
  J_{1} & 0 & 0 & 0 &    &  \\\cline{1-4}
  J_{2} &  &\mathcal{A}^{T} & & \mathcal{L}    & \\
\end{block}
\end{blockarray}\,,$}
\end{equation*}
for nowhere-zero vectors $\{\pmb{x}_{i}\}$ and a matrix $\mathcal{A}$, where $D(\mathcal{A})=0$ and $\mathcal{A}^{T}\pmb{\alpha}=\pmb{0}$, for every $\mu$-eigenvector $\pmb{\alpha}$ of $L_{\mathcal{H}}^{\rho}$,
then
$\displaystyle m_{L_{\mathcal{E}}^{\rho}}(\mu)=m_{L_{\mathcal{G}}^{\rho}}(\mu)+|I_{1}|$.
\end{theorem}
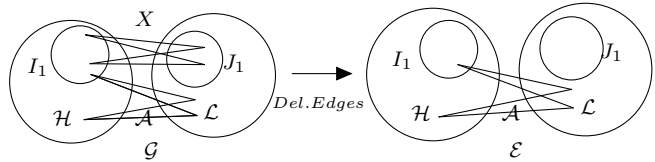
\begin{figure}[H]
\footnotesize
\centering
\begin{tikzpicture}[scale=0.6,line cap=round,line join=round,>=triangle 45,x=1.0cm,y=1.0cm]
\draw(-2.54,3.36) circle (1.36cm);
\draw(0.62,3.5) circle (1.37cm);
\draw(-2.34,3.96) circle (0.64cm);
\draw(0.2,3.86) circle (0.62cm);
\draw (0.42,4.12)-- (-2.22,4.4)-- (0.42,3.74)-- (-2.12,3.76);
\draw (0.42,4.12)-- (-2.12,3.76);
\draw (-3.12,2.9) node[anchor=north west] {$\mathcal{H}$};
\draw (-3.68,4.05) node[anchor=north west] {$I_1$};
\draw (0.64,4.14) node[anchor=north west] {$J_1$};
\draw (-1.3,5.1) node[anchor=north west] {$X$};
\draw (0.22,3.03) node[anchor=north west] {$\mathcal{L}$};
\draw [->] (2.36,3.54) -- (3.7,3.54);
\draw (-1.1,2.18) node[anchor=north west] {$\mathcal{G}$};
\draw(5.48,3.52) circle (1.47cm);
\draw(8.86,3.63) circle (1.47cm);
\draw(5.83,4.08) circle (0.64cm);
\draw(8.55,4.1) circle (0.7cm);
\draw (4.39,4.21) node[anchor=north west] {$I_1$};
\draw (4.83,3.0) node[anchor=north west] {$\mathcal{H}$};
\draw (8.57,3.14) node[anchor=north west] {$\mathcal{L}$};
\draw (9.0,4.39) node[anchor=north west] {$J_1$};
\draw (7.0,2.18) node[anchor=north west] {$\mathcal{E}$};
\draw (1.74,3.21) node[anchor=north west] {{\tiny $Del. Edges$}};
\draw (0.21,2.96)-- (-2.1,3.52);
\draw (-2.1,3.52)-- (0.25,2.61);
\draw (0.25,2.61)-- (-2.25,2.52);
\draw (-2.25,2.52)-- (0.25,2.61);
\draw (-2.1,3.52)-- (0.25,2.61);
\draw (0.21,2.96)-- (-2.25,2.52);
\draw (-1.3,2.9) node[anchor=north west] {$\mathcal{A}$};
\draw (5.61,2.58)-- (8.59,2.78);
\draw (8.59,2.78)-- (6.03,3.74);
\draw (6.03,3.74)-- (8.55,3.19);
\draw (8.55,3.19)-- (5.61,2.58);
\draw (6.83,3.0) node[anchor=north west] {$\mathcal{A}$};
\end{tikzpicture}
\caption{The graphs  of Theorem \ref{symmetricII}.}
\end{figure}

In Theorem \ref{symmetricII}, by putting $\mathcal{A}=0$, we conclude the following corollary for the (signless) Laplacian matrix of simple graphs:
\begin{corollary}\label{lap1}
Let $\mu\in \mathbb{R}$, $\rho\in \{-1,1\}$, and $H$ be a graph and $\{u_{1},\ldots,u_{t}\}$ be a subset of a $\mu$-star set of $L_{H}^{\rho}$. If $L$ is  an arbitrary graph disjoint from $H$,  and $G$ is the graph formed  by joining the vertex $u_{i}$ to an arbitrary vertex $v_{i}$ of  $L$ (not necessarily disjoint), $i\in [t]$, then $m_{L_{G}^{\rho}}(\mu)=m_{L_{L}^{\rho}}(\mu)+m_{L_{H}^{\rho}}(\mu)-t$.
\begin{figure}[H]
\centering
\begin{tikzpicture}[scale=0.8,line cap=round,line join=round,>=triangle 45,x=1.0cm,y=1.0cm]
\draw(-2.36,3.16) circle (1.2cm);
\draw(0.76,3.16) circle (1.2cm);
\draw (0.32,3.88)-- (-1.92,3.7);
\draw (0.4,2.74)-- (-1.82,2.94);
\draw (-2.9,3.38) node[anchor=north west] {$L$};
\draw (1.14,3.44) node[anchor=north west] {$H$};
\draw (-0.94,2.3) node[anchor=north west] {$G$};
\begin{scriptsize}
\draw [fill=black] (-1.92,3.7) circle (1.5pt);
\draw[] (-1.94,4.0) node {$v_1$};
\draw[] (-1.74,3.44) node {$\vdots$};
\draw [fill=black] (0.32,3.88) circle (1.5pt);
\draw[] (0.5,4.12) node {$u_1$};
\draw[] (0.2,3.42) node {$\vdots$};
\draw [fill=black] (0.4,2.74) circle (1.5pt);
\draw[] (0.58,2.98) node {$u_t$};
\end{scriptsize}
\end{tikzpicture}
\caption{Graph $G$ of Corollary \ref{lap1}.}
\end{figure}
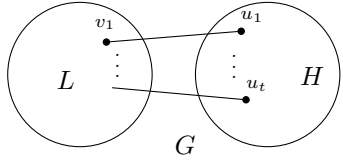
\end{corollary}

As a particular case of Corollary \ref{lap1} ($t=1$), we have this theorem of {\rm \cite{Ne}}:
\begin{corollary}{\rm \cite{Ne}}
Let  $G$  be  any graph  with a simple  Laplacian eigenvalue $\mu$.  Let  $u$ be 
a vertex of $G$  such that an eigenvector corresponding to $\mu$  is nonzero on $u$. 
Let $H$ be  any graph,  and  let $G^{\prime}$ be the graph formed  by joining  an arbitrary 
vertex of  $H$ to $u$.  Then $m_{L(H)}(\mu)=m_{L(G^{\prime})}(\mu)$.
\end{corollary}

\begin{proof}
Assume that 
$L_{G}^{\rho}=\begin{blockarray}{cc|c}
\begin{block}{c(c|c)}
\text{\footnotesize $u$ } & a  & \pmb{x}^{T}   \\\cline{1-3}
\text{\tiny $V(G)-\{u\}$ }&\pmb{x}  &  M  \\
\end{block}
\end{blockarray}
$
 and $\pmb{\alpha}$ is a $\mu$-eigenvector of  $L_{G}^{\rho}$ such that $\pmb{\alpha}(u)\neq 0$. It is sufficient to show that $\{u\}$ is  a $\mu$-star set of $L_{G}^{\rho}$. On the other hand, we show $m_{M}(\mu)=0$. Suppose, by contradiction, $M$ has a $\mu$-eigenvector $\pmb{\beta}$. If $\pmb{x}^{T}\pmb{\beta}=0$, then $\pmb{y}$ is a $\mu$-eigenvector of  $L_{G}^{\rho}$, where $\pmb{y}(v)=
\begin{cases}
\pmb{\beta}(v)& v\neq u,\\
0 & v=u
\end{cases}$.
Since $\pmb{\alpha}(u)\neq 0$, the vectors $\pmb{\alpha}$ and $\pmb{y}$ are independent and we have a contradiction with $m_{L_{G}^{\rho}}(\mu)=1$. If $\pmb{x}^{T}\pmb{\beta}\neq 0$, then
\[
0=(\mu\mathbb{I}-L_{G}^{\rho})\pmb{\alpha}\Rightarrow \pmb{\alpha}(u)\,\pmb{x}=(\mu\mathbb{I}-M)\pmb{\alpha}_{|\text{\tiny $V(G)-\{u\}$ }}\Rightarrow
\pmb{\alpha}(u)\pmb{\beta}^{T}\pmb{x}=\pmb{\beta}^{T}(\mu\mathbb{I}-M)\pmb{\alpha}_{|\text{\tiny $V(G)-\{u\}$ }}=0 \xRightarrow{\pmb{\beta}^{T}\pmb{x}\neq 0}
\pmb{\alpha}(u)=0,
\]
 and we have a contradiction. This completes the proof.
\end{proof}

\begin{remark}
By Corollary \ref{lap1}, since $m_{L(P_3)}(1)=1$ and the value of a $1$-eigenvector is nonzero on every pendant vertex of $P_3$,     we have Theorem \ref{delp3}. Also,  $m_{L(S_k)}(k)=1$ and every $k$-eigenvector of $S_k$ is nowhere-zero, hence we have Theorem \ref{delsn}.
\end{remark}

\subsection{Edge Switching}

In the following theorem, for a given eigenvalue $\mu$, a particular subgraph, and  given weights of the edges, we delete some edges and switch some weights from a section of graph to another section and give the relation between the multiplicity of $\mu$ for two graphs.

\begin{theorem}\label{switching}
Let $\mu\in \mathbb{R},\,\rho\in \mathbb{R}-\{0\}$ and $\mathcal{H},\mathcal{L}$ be real symmetric matrices  with row and column indices $I=I_{1}\dot{\cup} I_{2}\dot{\cup} I_{3}$ and $J=J_{1}\dot{\cup} J_{2}$, respectively. Suppose that $S\in \text{\rm Sym}_{|I_{1}|}(\mathbb{R})$ and $I_{1}\dot{\cup} I_{2}$ is a $\mu$-star set of $L_{\mathcal{H}}^{\rho}$. If  $\widehat{\mathcal{H}},\widehat{\mathcal{L}},\mathcal{G},$ and $\mathcal{E}$ are symmetric matrices given below,
\begin{equation*}
\resizebox{ 0.8\textwidth}{!} 
{$
\newcommand*{\temp}{\BAmulticolumn{1}{|c}{}}
\newcommand*{\tempa}{\BAmulticolumn{1}{|c}{\mathcal{A}}}
\newcommand*{\tempj}{\BAmulticolumn{1}{|c}{J_{2}}}
\widehat{\mathcal{H}}=\mathcal{H}+
\begin{blockarray}{c|c}
  I_{1} & I\setminus I_{1}\\
 \begin{block}{(c|c)}
   S & 0  \\ \cline{1-2}
   0 & 0 \\
 \end{block}
\end{blockarray}\,
,\,\,
\widehat{\mathcal{L}}
=\mathcal{L}-
\begin{blockarray}{c|c}
 J_{1} &J_{2}\\
 \begin{block}{(c|c)}
   S' &  0   \\\cline{1-2}
    0  & 0 \\
 \end{block}
\end{blockarray}\,,\,\,
\mathcal{G}=
\begin{blockarray}{cccc|cc}
 & I_{1} & I_{2} & I_{3}  & J_{1}  &\tempj \\
\begin{block}{c(ccc|cc)}
  I_{1} & & & &  X  &\temp  \\
  I_{2} & & \widehat{\mathcal{H}}& &  0&\tempa\\
  I_{3} &  & & & 0 &\temp  \\\cline{1-6}
  J_{1} & X^{T} & 0 & 0 &    &  \\\cline{1-4}
  J_{2} & & \mathcal{A}^{T} &  &     \mathcal{L} & \\
\end{block}
\end{blockarray}
\,,\,\,
\mathcal{E}=
\begin{blockarray}{cccc|cc}
 & I_{1} & I_{2} & I_{3}  & J_{1}  &\tempj \\
\begin{block}{c(ccc|cc)}
  I_{1} & & & &  0  &\temp  \\
  I_{2} & & \mathcal{H}& &  0&\tempa\\
  I_{3} &  & & & 0 &\temp  \\\cline{1-6}
  J_{1} & 0 & 0 & 0 &    &  \\\cline{1-4}
  J_{2} & & \mathcal{A}^{T} &  &     \widehat{\mathcal{L}} & \\
\end{block}
\end{blockarray}\,,$}
\end{equation*}
for some matrices $\mathcal{A}$ and $X$, where  $D(\mathcal{A})=0$, $\mathcal{A}^{T}\pmb{\alpha}=\pmb{0}$, for every $\mu$-eigenvector $\pmb{\alpha}$ of $\mathcal{H}$, and $S'$ is a solution of the equation $L_{S'}^{\rho}+D(X^{T})=\rho^{2}X^{T} (L_{S}^{\rho}+D(X))^{-1}X$,
then
$m_{L_{\mathcal{E}}^{\rho}}(\mu)=m_{L_{\mathcal{G}}^{\rho}}(\mu)+|I_{1}|.$

In particular, if $L_{S}^{\rho}$ is invertible, then $m_{L_{\widehat{\mathcal{H}}}^{\rho}}(\mu)=m_{L_{\mathcal{H}}^{\rho}}(\mu)-|I_{1}|$.
\end{theorem}

For any $\rho\in\mathbb{R}\setminus\{0,1\}$ and $L\in \text{\rm Sym}_{n}(\mathbb{R})$, it is easy to see that the equation $L=L_{M}^{\rho}$ has a unique solution $M\in \text{\rm Sym}_{n}(\mathbb{R})$. Thus, for given $S$ ans $X$ such that there exists $(L_{S}^{\rho}+D(X))^{-1}$,  the equation {\small $L_{S'}^{\rho}+D(X^{T})=\rho^{2}X^{T} (L_{S}^{\rho}+D(X))^{-1}X$} has a solution for $S'$.

\begin{corollary}[Edge Switching]\label{switch}
Let $\rho^{2}=1$ and $L,H$ be two disjoint graphs. With the notations of  Theorem \ref{switching}, put  $\mathcal{H}=A(H)$, $\mathcal{L}=A(L)$, $\mathcal{A}=0$, and Consider the following two cases:
\begin{enumerate}
\item
$X=\mathbb{I}_{|I_{1}|}$ and $-S$ is a permutation matrix corresponding to an involution,
\item
for a given $S$, suppose that $X$ is a solution of $X=L_{S}^{\rho}+D(X)$. 
\end{enumerate}
Then, for both cases, $S'=S$ is a solution and $m_{L_{G}^{\rho}}(\mu)+|I_{1}|=m_{L_{H}^{\rho}}(\mu)+m_{L_{\widehat{L}}^{\rho}}(\mu)$.

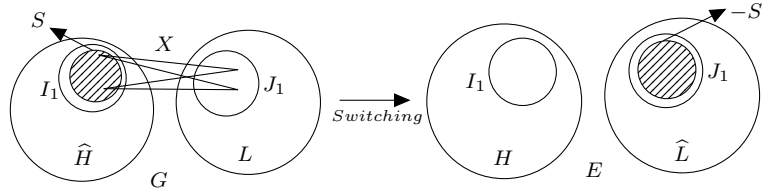
\begin{figure}[H]
\footnotesize
\centering
\begin{tikzpicture}[scale=0.7,line cap=round,line join=round,>=triangle 45,x=1.0cm,y=1.0cm]
\draw(-2.54,3.36) circle (1.36cm);
\draw(0.62,3.5) circle (1.37cm);
\draw(-2.34,3.96) circle (0.64cm);
\draw(0.2,3.86) circle (0.62cm);
\draw (0.42,4.12)-- (-2.22,4.4)-- (0.42,3.74)-- (-2.12,3.76);
\draw (0.42,4.12)-- (-2.12,3.76);
\draw [fill=black,pattern=north east lines,pattern color=black] (-2.28,4.) circle (0.49cm);
\draw [->] (-2.34,4.49) -- (-3.16,4.92);
\draw (-3.67,5.36) node[anchor=north west] {$S$};
\draw (-2.86,2.86) node[anchor=north west] {$\widehat{H}$};
\draw (-3.5,4.05) node[anchor=north west] {$I_1$};
\draw (0.74,4.14) node[anchor=north west] {$J_1$};
\draw (-1.32,4.90) node[anchor=north west] {$X$};
\draw (0.26,2.82) node[anchor=north west] {$L$};
\draw [->] (2.36,3.54) -- (3.7,3.54);
\draw (-1.4,2.32) node[anchor=north west] {$G$};
\draw(5.48,3.52) circle (1.47cm);
\draw(8.86,3.63) circle (1.47cm);
\draw(5.83,4.08) circle (0.64cm);
\draw(8.55,4.1) circle (0.7cm);
\draw [fill=black,pattern=north east lines,pattern color=black] (8.57,4.12) circle (0.55cm);
\draw [->] (8.5,4.67) -- (9.7,5.28);
\draw (4.59,4.21) node[anchor=north west] {$I_1$};
\draw (5.11,2.75) node[anchor=north west] {$H$};
\draw (8.55,2.94) node[anchor=north west] {$\widehat{L}$};
\draw (9.61,5.54) node[anchor=north west] {$-S$};
\draw (9.17,4.39) node[anchor=north west] {$J_1$};
\draw (6.86,2.52) node[anchor=north west] {$E$};
\draw (2.04,3.51) node[anchor=north west] {{\tiny $Switching$}};
\end{tikzpicture}
\caption{A schematic diagram of graphs $G$ and $E$ of Corollary \ref{switch}.}
\end{figure}
\end{corollary}

Since for a non-bipartite graph $H$, the signless Laplacian matrix $Q(H)$ is invertible, by a particular case of Theorem \ref{switching}, if we set $S=\pm\bigl( \begin{smallmatrix}
0 & 1 \\
1 & 0
\end{smallmatrix}\bigr)$, then we can conclude the next corollary.
\begin{corollary}
Let $H$ be a non-bipartite graph with a $\mu$-star set $U$, for a signless Laplacian eigenvalue  $\mu$. If $u,v\in U$ and $uv\in E(H)$
{\rm(}$uv\notin E(H)${\rm)}, then removing {\rm(}adding, resp.{\rm)} the edge $uv$, decreases $m_{Q(H)}(\mu)$ by $2$ and $U\setminus \{u,v\}$ is a $\mu$-star set of $Q(\widehat{H})$.
\end{corollary}

Suppose that $G$ is an $r$-regular graph and $\lambda\in\mathbb{R}$. If $m_{A(G)}(\lambda)=k$ and $U$ is a $\lambda$-star set  of $A(G)$, then $U$ is an $(r-\rho \lambda)$-star set of $L_{G}^{\rho}$. So, we can conclude some results of Laplacian matrices similar to adjacency matrices of regular subgraphs. For example, we state the following corollary on path subgraphs for Laplacian matrices such that there is in \cite{BK} for adjacency matrices of graphs:

\begin{corollary}\label{delpn}
A path with $n$ vertices of valency $2$ in a graph $G$ can be replaced by an edge {\rm(}see Figure \ref{pathn}{\rm)} without changing the multiplicity of  Laplacian eigenvalue $4\sin^2 (\frac{k\pi}{n})$ (signless Laplacian eigenvalue $4\cos^2 (\frac{k\pi}{n})$), for $n\geq 3$ and $k\in [n-1],k\neq \frac{n}{2}$.
\begin{figure}[H]
\centering
\begin{tikzpicture}[node distance=0.7 cm, >=stealth',
	minimum size=2 pt, inner sep=1 pt, line width=0.6 pt]
\tikzstyle{init} = [pin edge={to-, thin, white}]
\tikzstyle{place}=[circle, draw ,thick,fill=black]
\tikzstyle{label}=[circle , minimum size=1 pt,thick]

\node [place] (v) at (0,0) {};
\node [label] (vn) at (0,0.3)  {$v$};
\node [label] (h1) at (-0.3,-0.4) {};
\node [label] (h2) at (0,-0.4) {};
\node [label] (h3) at (0.3,-0.4) {};

\node [place] (u1) at (0.5,0.5)  {};
\node [label] (u1n) at (0.5,0.8)  {$1$};
\node [place] (u2) at (1,0.5)  {};
\node [label] (u2n) at (1,0.8)  {$2$};
\node [place] (u3) at (1.5,0.5)  {};
\node [place] (u4) at (2,0.5)  {};
\node [label] (u4n) at (2,0.8)  {$n$};

\node [place] (u) at (2.5,0)  {};
\node [label] (un) at (2.5,0.3)  {$u$};
\node [label] (h4) at (2.2,-0.4) {};
\node [label] (h5) at (2.5,-0.4) {};
\node [label] (h6) at (2.8,-0.4) {};

\node [place] (vv) at (4,0)  {};
\node [label] (vvn) at (4,0.3)  {$v$};
\node [label] (hh1) at (3.7,-0.4) {};
\node [label] (hh2) at (4,-0.4) {};
\node [label] (hh3) at (4.3,-0.4) {};

\node [place] (uu) at (5,0)  {};
\node [label] (uun) at (5,0.3)  {$u$};
\node [label] (hh4) at (4.7,-0.4) {};
\node [label] (hh5) at (5,-0.4) {};
\node [label] (hh6) at (5.3,-0.4) {};

\draw[-] (v) -- (u1);
\draw[-] (v) -- (h1);
\draw[-] (v) -- (h2);
\draw[-] (v) -- (h3);
\draw[-] (u) -- (u4);
\draw[-] (u) -- (h4);
\draw[-] (u) -- (h5);
\draw[-] (u) -- (h6);

\draw[-] (u1) -- (u2);
\draw[dotted] (u2) -- (u3);
\draw[-] (u3) -- (u4);

\draw[-] (vv) -- (uu);
\draw[-] (vv) -- (hh1);
\draw[-] (vv) -- (hh2);
\draw[-] (vv) -- (hh3);
\draw[-] (uu) -- (hh4);
\draw[-] (uu) -- (hh5);
\draw[-] (uu) -- (hh6);

\end{tikzpicture}
\caption{The graphs of Corollary \ref{delpn}. }
\label{pathn}
\end{figure}
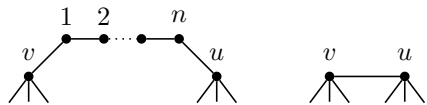
\end{corollary}

\begin{proof}
We prove the corollary for $\rho=1$, the other case is similar. The eigenvalues of $L_{C_{n}}^{1}$ are $(4{\sin^{2}(\frac{k\pi}{n})})_{k=0}^{n-1}$ {\rm(}see \cite{Sp}{\rm)}. If $\mu=4{\sin^{2}(\frac{k\pi}{n})}$ for an integer $k$, $k\in [n-1]$ and $k\neq \frac{n}{2}$, then $m_{L_{C_{n}}^{1}}(\mu)=2$. Since every set of two adjacent vertices of $C_{n}$ is a $(2-\mu)$-star set of $A(C_{n})$ \cite{BK}, so it is a $\mu$-star set of $L_{C_{n}}^{1}$. Thus, by Corollary \ref{switch}, the proof is complete.
\end{proof}

\begin{example}In Corollary \ref{delpn}:
\begin{itemize}
\item
$n=3$ and $k=1$: A path with three vertices of valency $2$ in a graph $G$ can be replaced by an edge, without changing the Laplacian multiplicity of $3$ and the signless Laplacian multiplicity of $1$.
\item
$n=4$ and $k=1$: A path with four vertices of valency $2$ in a graph $G$ can be replaced by an edge, without changing the (signless) Laplacian multiplicity of $2$.
\end{itemize}
\end{example}

\section{Type II Reductions: Deleting Subgraphs}
In this section we generalize  Theorem \ref{dcluster} for a given generalized Laplacian eigenvalue $\mu$.
\begin{theorem}\label{cluster}
Let $\mu\in \mathbb{R},\rho\in \mathbb{R}-\{0\},r\in\mathbb{N}$, and $\mathcal{H},\{\mathcal{H}_i\}_{i=1}^{r-1},\mathcal{L},\{\mathcal{L}_i\}_{i=1}^{r-1},\mathcal{K}$ be symmetric matrices. Suppose that $L_{\mathcal{E}_{i}}^{\rho}$ has a $\mu$-eigenvector $\pmb{\gamma}^{i}$ such that ${\pmb{\gamma}^{i}}_{|J}=\pmb{0}$ and ${\pmb{\gamma}^{i}}_{|I_{i}}\neq \pmb{0}$, $i\in [r-1]$, and $L_{\mathcal{E}}^{\rho}$ has independent $\mu$-eigenvectors $\pmb{\beta}^{i}$ such that ${\pmb{\beta}^{i}}_{|J}=\pmb{0}$, $i\in [s]$. Then $m_{L_{\mathcal{G}}^{\rho}}(\mu)\geq s+r-1$, where
\begin{equation*}
\resizebox{ 0.8\textwidth}{!} 
{$
\mathcal{E}_{i}=
\begin{blockarray}{cc|c|c}
\begin{block}{c(c|c|c)}
  I & \mathcal{H} &\mathcal{A}  &0  \\\cline{1-4}
  J & \mathcal{A}^{T}& \mathcal{L}_{i}&B_{i}\\\cline{1-4}
  I_{i} & 0& B_{i}^{T} & \mathcal{H}_{i} \\
\end{block}
\end{blockarray},\,\, i\in[r-1],\,\,
\mathcal{E}=
\begin{blockarray}{c|c|c}
\begin{block}{(c|c|c)}
  \mathcal{H} &\mathcal{A}  &0  \\\cline{1-3}
  \mathcal{A}^{T}& \mathcal{L}&B\\\cline{1-3}
 0& B^{T} & \mathcal{K} \\
\end{block}
\end{blockarray},\,\, 
\mathcal{G}=
\begin{blockarray}{cc|c|c|c|c|c}
\begin{block}{c(c|c|c|c|c|c)}
I & \mathcal{H}   & 0 &\cdots&0& \mathcal{A} &0 \\\cline{1-7}
I_{1} &0&  \mathcal{H}_{1} &  \cdots&0& B_{1} &0 \\\cline{1-7}
\vdots&\vdots& \vdots  &\ddots  &\vdots&\vdots&\vdots \\\cline{1-7}
I_{r-1} &  0 & 0 &  \cdots & \mathcal{H}_{r-1} &  B_{r-1}&0 \\\cline{1-7}
J& \mathcal{A}^{T} &   B_{1}^{T} &  \cdots& B_{r-1}^{T}& \mathcal{L} & B  \\\cline{1-7}
J_{1}& 0 &   0 &  \cdots& 0& B^{T}&\mathcal{K}  \\
\end{block}
\end{blockarray}\,.$}
\end{equation*}
\end{theorem}
\begin{figure}[H]
\footnotesize
\centering
\begin{tikzpicture}[scale=0.6,line cap=round,line join=round,>=triangle 45,x=1.0cm,y=1.0cm]
\draw(-1.21,3.43) circle (0.9cm);
\draw(0.96,3.47) circle (0.71cm);
\draw(3.1,3.39) circle (0.9cm);
\draw(5.38,0.34) circle (0.57cm);
\draw(5.52,3.88) circle (0.66cm);
\draw(5.47,2.22) circle (0.676cm);
\draw(-1.12,0.76) circle (0.92cm);
\draw(1.18,0.69) circle (0.79cm);
\draw(3.19,0.65) circle (0.7cm);
\draw (-3.1,3.9) node[anchor=north west] { $\mathcal{E}$};
\draw (-3.1,1.2) node[anchor=north west] { $\mathcal{E}_{i}$};
\draw (8.5,0.7) node[anchor=north west] { $\mathcal{G}$};
\draw (-1.68,3.9) node[anchor=north west] { $\mathcal{H}$};
\draw (-1.65,1.2) node[anchor=north west] {$\mathcal{H}$};
\draw (0.56,3.9) node[anchor=north west] {$\mathcal{L}$};
\draw (0.71,1.23) node[anchor=north west] {$\mathcal{L}_i$};
\draw (2.89,3.90) node[anchor=north west] {$\mathcal{K}$};
\draw (3.01,1.15) node[anchor=north west] {$\mathcal{H}_i$};
\draw (5.07,4.36) node[anchor=north west] {$\mathcal{H}$};
\draw (4.79,2.63) node[anchor=north west] {$\mathcal{H}_1$};
\draw (4.83,2.03) node[anchor=north west] {$\vdots$};
\draw (4.74,0.71) node[anchor=north west] {$\mathcal{H}_{r-1}$};
\draw (7.34,2.3) node[anchor=north west] {$\mathcal{L}$};
\draw (10.12,2.38) node[anchor=north west] {$\mathcal{K}$};
\draw (-0.87,3.81)-- (0.83,3.79)-- (-0.85,3.16)-- (0.67,3.21);
\draw (1.27,3.72)-- (3.03,3.76)-- (1.34,3.23)-- (2.9,3.14);
\draw (-0.61,1.09)-- (0.87,1.09)-- (-0.61,0.52)-- (0.85,0.49);
\draw (1.61,0.94)-- (3.25,0.92)-- (1.52,0.54)-- (3.12,0.49);
\draw (-0.61,1.09)-- (0.85,0.49);
\draw (-0.87,3.81)-- (0.67,3.21);
\draw (1.27,3.72)-- (2.85,3.14);
\draw (1.63,0.94)-- (3.12,0.49);
\draw (-0.44,4.50) node[anchor=north west] {$\mathcal{A}$};
\draw (-0.28,1.82) node[anchor=north west] {$\mathcal{A}$};
\draw (1.44,4.44) node[anchor=north west] {$B$};
\draw (1.79,1.74) node[anchor=north west] {$B_i$};
\draw (6.5,4.19) node[anchor=north west] {$\mathcal{A}$};
\draw (5.81,2.00) node[anchor=north west] {$B_1$};
\draw (6.32,0.79) node[anchor=north west] {$B_{r-1}$};
\draw [rotate around={-3.17:(10.25,2.01)}] (10.25,2.01) ellipse (1.09cm and 0.94cm);
\draw (7.58,2.53)-- (5.93,4.03)-- (8.17,2.62)-- (5.68,3.57);
\draw(7.74,1.98) circle (1.12cm);
\draw (5.68,3.57)-- (7.58,2.53);
\draw (5.74,2.53)-- (7.21,2.19);
\draw (7.21,2.19)-- (5.74,2.01);
\draw (5.74,2.01)-- (7.03,1.76);
\draw (5.74,2.5)-- (7.03,1.76);
\draw (5.56,0.69)-- (7.31,1.36);
\draw (5.56,0.69)-- (7.31,1.36);
\draw (5.56,0.69)-- (7.77,1.06);
\draw (7.77,1.06)-- (5.68,0.23);
\draw (5.68,0.23)-- (7.31,1.36);
\draw (8.41,2.41)-- (9.98,2.38);
\draw (8.41,2.41)-- (10.13,1.55);
\draw (10.13,1.55)-- (8.41,1.58);
\draw (8.41,1.58)-- (9.98,2.38);
\draw (8.56,3.05) node[anchor=north west] {$B$};
\end{tikzpicture}
\caption{A schematic figure of Theorem \ref{cluster}. }
\end{figure}
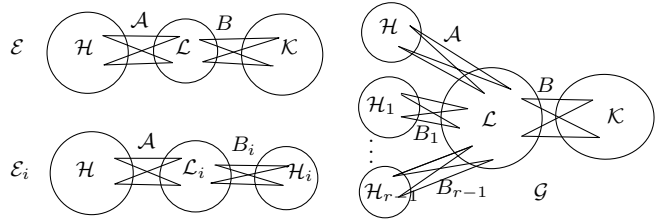
Now, we can conclude the following theorem of \cite{GM} on $d$-clusters.
\begin{corollary}{\rm \cite{GM}}
 Let $G$ be a graph with $k$ independent $d$-clusters of orders $r_{1},\ldots,r_{k}$. Then $m_{L(G)}(d)\geq \sum_{i=1}^{k} r_i-k$.
 \end{corollary}
\begin{proof}
With the notations of Theorem \ref{cluster}, we put 
$\rho=1, \mathcal{H}=\mathcal{H}_{i}=[0]_{1\times 1}$, $\mathcal{L}_{i}=[0]_{d\times d}$, and  $\mathcal{A}=B_{i}^{T}=\pmb{j}_{d}$, then $m_{L_{\mathcal{E}_i}^{\rho}}(d)=1$, for $d\neq 2$, and $m_{L_{\mathcal{E}_i}^{\rho}}(d)=2$, for $d=2$,  and $\pmb{\gamma}^{i}=(1,0,\ldots,0,-1)^{T}$ is a $d$-eigenvector. Since $d$-clusters are independent, by using Theorem \ref{cluster}, $k$ times, we have 
$m_{L_{G}^{\rho}}(d)\geq \sum_{i=1}^{k} (r_i-1)= \sum_{i=1}^{k} r_i-k$.
\end{proof}
Now, we give our first proof for Theorem \ref{pkqk}. First, we need the following lemma on eigenvectors of the path graph.
\begin{lemma}{\rm \cite{Sp}}\label{pathvec}
Let $n$ be a positive integer. Then $\displaystyle 4\cos^{2}(\frac{j\pi}{2n})$ for $j\in [n]$ {\rm (}$\displaystyle 4\sin^{2}(\frac{l\pi}{2n})$ for $0\leq l\leq n-1${\rm )} is a Laplacian eigenvalue of $P_n$ with the corresponding eigenvector $\pmb{v}_j$, where $\pmb{v}_{j}(u)=\cos(\frac{(n-j)(2u-1)\pi}{2n})$, for $u\in [n]$.
\end{lemma}
Since the signless Laplacian matrix and the Laplacian matrix of a path are similar, it is easy to see that $\pmb{w}_{j}$ is a  signless Laplacian eigenvector corresponding to  $4\cos^{2}(\frac{j\pi}{2n})$, where $\pmb{w}_{j}(u)=(-1)^{u}\pmb{v}_{j}(u)$, for $j,u\in [n]$.\\

{\noindent \textbf{First proof of Theorem \ref{pkqk}:}
By Lemma \ref{pathvec}, if $n=2k+1$, $j=2t$, and $u=k+1$, then $\pmb{v}_{j}(u)=0$. With the notations of Theorem \ref{cluster}, we put $\mu=4\cos^{2}(\frac{2t\pi}{2(2k+1)})$, $\rho=1, \mathcal{H}=\mathcal{H}_{i}=A(P_{k})$, $\mathcal{L}_{i}=[0]_{1\times 1}$, $\mathcal{A}=\pmb{e}_{k}$, and $B_{i}^{T}=\pmb{e}_{1}$, then we have $\mathcal{E}_{i}=A(P_{2k+1})$. If we put $\pmb{\gamma}^{i}=\pmb{v}_{j}$, then by using Theorem \ref{cluster}, $q_{k}(G)$ times, we have $m_{L_{G}^{\rho}}(4\cos^{2}(\frac{t\pi}{2k+1}))\geq p_{k}(G)-q_{k}(G)$, $t\in [k]$. By similar proof, we have the statement is true for $\rho=-1$.
}

\section{Type III Reductions: Splitting Vertices}

In this section, we state a splitting method to simplify graphs for a generalized Laplacian eigenvalue $\mu$ and a particular subgraph corresponding to it.

\begin{theorem}\label{spl}
Let $\mu\in \mathbb{R},\,\rho\in \mathbb{R}-\{0\}$ and $\mathcal{H},\mathcal{L}$ be real symmetric matrices. Suppose that $m_{L_{\mathcal{H}}^{\rho}+D(\pmb{x})}(\mu)=1$. If $\pmb{x}^{T}\pmb{\alpha}\neq 0$, for a  $\mu$-eigenvector $\pmb{\alpha}$ of $L_{\mathcal{H}}^{\rho}+D(\pmb{x})$, then $m_{L_{\mathcal{G}}^{\rho}}(\mu)=m_{L_{\mathcal{E}}^{\rho}}(\mu)$, where 
\begin{equation*}
\resizebox{ 0.75\textwidth}{!} 
{$
\mathcal{G}=
\begin{blockarray}{cc|c|c}
\begin{block}{c(c|c|c)}
  I & \mathcal{H} & \pmb{x} & 0 \\\cline{1-4}
  v & \pmb{x}^{T} & a &  \pmb{y}^{T}\\\cline{1-4}
  J & 0 &  \pmb{y} &\mathcal{L}  \\
\end{block}
\end{blockarray}\,,\,\,\,
\pmb{y}=
\begin{blockarray}{c}
\begin{block}{(c)}
  y_1 \\
  y_2 \\
  \vdots \\
  y_{|J|} \\
\end{block}
\end{blockarray}\,,\,\,\,
\mathcal{E}=
\begin{blockarray}{cc|c|c|c|c}
\begin{block}{c(c|c|c|c|c)}
\begin{array}{c}
I_{1} \\ 
v_{1}
\end{array}& \begin{array}{c|c}
 \mathcal{H} &  \pmb{x} \\\cline{1-2}
 \pmb{x}^{T} & a 
\end{array}  & 0 &\cdots&0& \begin{array}{c}
0 \\ \cline{1-1}
y_{1}\pmb{e}^{T}_{1}
\end{array}  \\\cline{1-6}
\begin{array}{c}
I_{2} \\ 
v_{2}
\end{array}&0&\begin{array}{c|c}
 \mathcal{H} &  \pmb{x} \\\cline{1-2}
 \pmb{x}^{T} & a 
\end{array}  &  \cdots&0& \begin{array}{c}
0 \\ \cline{1-1}
y_{2}\pmb{e}^{T}_{2}
\end{array}\\\cline{1-6}
\vdots&\vdots& \vdots  &\ddots  &\vdots&\vdots\\\cline{1-6}
\begin{array}{c}
I_{|J|} \\ 
v_{|J|}
\end{array}&  0 & 0 &  \cdots &\begin{array}{c|c}
 \mathcal{H} &  \pmb{x} \\\cline{1-2}
 \pmb{x}^{T} & a 
\end{array}& \begin{array}{c}
0\\ \cline{1-1}
y_{|J|}\pmb{e}^{T}_{|J|}
\end{array}\\\cline{1-6}
J&  \begin{array}{c|c}
0& y_{1}\pmb{e}_{1}
\end{array} &    \begin{array}{c|c}
0& y_{2}\pmb{e}_{2}
\end{array} &  \cdots& \begin{array}{c|c}
0& y_{|J|}\pmb{e}_{|J|}
\end{array}&
  \mathcal{L}  \\
\end{block}
\end{blockarray}\,,$}
\end{equation*}
\end{theorem}
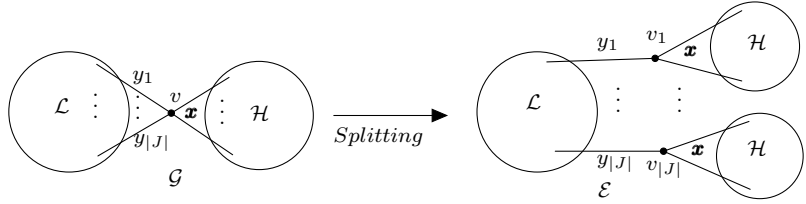
\begin{figure}[H]
\centering
\footnotesize
\begin{tikzpicture}[scale=0.8,line cap=round,line join=round,>=triangle 45,x=1.0cm,y=0.99cm]
\draw(-3.46,2.56) ellipse (1.cm and 1.cm);
\draw(-0.3,2.5) ellipse (0.9cm and 0.9cm);
\draw (-1.76,2.54)-- (-0.8,3.14);
\draw (-1.76,2.54)-- (-0.74,1.84);
\draw (-1.76,2.54)-- (-3.,3.36);
\draw (-1.76,2.54)-- (-2.98,1.82);
\draw (-2.54,3.38) node[anchor=north west] {$y_1$};
\draw (-3.20,3.26) node[anchor=north west] {$\vdots$};
\draw (-2.50,3.23) node[anchor=north west] {$\vdots$};
\draw (-2.52,2.33) node[anchor=north west] {$y_{|J|}$};
\draw (-1.66,2.76) node[anchor=north west] {$\pmb{x}$};
\draw (-3.84,2.92) node[anchor=north west] {$\mathcal{L}$};
\draw (-0.58,2.8) node[anchor=north west] {$\mathcal{H}$};
\draw (-1.10,3.16) node[anchor=north west] {$\vdots$};
\draw [->] (0.94,2.5) -- (2.84,2.5);
\draw (-1.92,3.02) node[anchor=north west] {$v$};
\draw(4.32,2.56) ellipse (1.cm and 1cm);
\draw (6.28,3.46)-- (4.48,3.4);
\draw (6.42,1.9)-- (4.62,1.9);
\draw(7.94,3.66) ellipse (0.74cm and 0.74cm);
\draw(8.02,1.82) ellipse (0.72cm and 0.71cm);
\draw (6.28,3.46)-- (7.72,4.26);
\draw (6.28,3.46)-- (7.72,3.08);
\draw (6.42,1.9)-- (7.84,2.4);
\draw (6.42,1.9)-- (7.86,1.26);
\draw (4.0,2.98) node[anchor=north west] {$\mathcal{L}$};
\draw (7.68,4) node[anchor=north west] {$\mathcal{H}$};
\draw (7.68,2.22) node[anchor=north west] {$\mathcal{H}$};
\draw (6.64,3.8) node[anchor=north west] {$\pmb{x}$};
\draw (6.76,2.14) node[anchor=north west] {$\pmb{x}$};
\draw (5.18,3.94) node[anchor=north west] {$y_1$};
\draw (5.48,3.36) node[anchor=north west] {$\vdots$};
\draw (6.5,3.36) node[anchor=north west] {$\vdots$};
\draw (5.18,1.98) node[anchor=north west] {$y_{|J|}$};
\draw (0.8,2.44) node[anchor=north west] {$Splitting$};
\draw (6.,4.05) node[anchor=north west] {$v_{1}$};
\draw (6.,1.88) node[anchor=north west] {$v_{|J|}$};
\draw (-1.94,1.72) node[anchor=north west] {$\mathcal{G}$};
\draw (5.2,1.5) node[anchor=north west] {$\mathcal{E}$};
\begin{scriptsize}
\draw [fill=black] (-1.76,2.54) circle (1.5pt);
\draw [fill=black] (6.28,3.46) circle (1.5pt);
\draw [fill=black] (6.42,1.9) circle (1.5pt);
\end{scriptsize}
\end{tikzpicture}
\caption{The splitting method: Theorem \ref{spl}. }
\end{figure}

\begin{lemma}\label{splitlem}
Let $n\in\mathbb{N},\, a,\mu\in \mathbb{R},\,\rho\in \mathbb{R}-\{0\}, \,\mathcal{H}\in \text{\rm Sym}_{n}(\mathbb{R})$, and $\pmb{x}\in\mathbb{R}^{n}$. The following statements are equivalent:
\begin{enumerate}[(i)]
\item
$m_{L_{\mathcal{H}}^{\rho}+D(\pmb{x})}(\mu)=1$ and $\pmb{x}^{T}\pmb{\alpha}\neq 0$, for a  $\mu$-eigenvector $\pmb{\alpha}$ of $L_{\mathcal{H}}^{\rho}+D(\pmb{x})$;
\item
$m_{L_{\mathcal{K}}^{\rho}}(\mu)=1$ and $\pmb{\beta}(v)=0$, for a  $\mu$-eigenvector $\pmb{\beta}$ of $L_{\mathcal{K}}^{\rho}$.
\end{enumerate}
Furthermore, (i) and (ii) imply $m_{L_{\widehat{\mathcal{H}}}^{\rho}}(\mu)=0$, where
\[
\mathcal{K}=
\begin{blockarray}{cc|c|c}
\begin{block}{c(c|c|c)}
I&\mathcal{H} & \pmb{x} &0  \\\cline{1-4}
v&\pmb{x}^{T} & a &  \pmb{x}^{T}\\\cline{1-4}
I'&0 &  \pmb{x} &\mathcal{H}  \\
\end{block}
\end{blockarray}
\quad
\text{and}\quad
\widehat{\mathcal{H}}=
\begin{blockarray}{cc|c}
\begin{block}{c(c|c)}
I&  \mathcal{H} & \pmb{x} \\\cline{1-3}
v&  \pmb{x}^{T} & a \\
\end{block}
\end{blockarray}\,.
\]
\end{lemma}

\begin{figure}[H]
\footnotesize
\centering
\begin{tikzpicture}[scale=0.7,line cap=round,line join=round,>=triangle 45,x=1.0cm,y=1.0cm]
\draw(-3.25,2.56) circle (0.91cm);
\draw(-0.25,2.5) circle (0.91cm);
\draw (-1.76,2.54)-- (-0.8,3.14);
\draw (-1.76,2.54)-- (-0.74,1.84);
\draw (-1.76,2.54)-- (-2.74,3.2);
\draw (-1.76,2.54)-- (-2.76,1.9);
\draw (-1.66,2.86) node[anchor=north west] {$\pmb{x}$};
\draw (-3.64,2.9) node[anchor=north west] {$\mathcal{H}$};
\draw (-0.61,2.9) node[anchor=north west] {$\mathcal{H}$};
\draw (-2.,3.12) node[anchor=north west] {$v$};
\draw(2.48,2.54) circle (0.91cm);
\draw (4,2.67) node[anchor=north west] {$v$};
\draw (2.9,3.25)-- (4.16,2.66);
\draw (4.16,2.66)-- (2.9,1.8);
\draw (3.3,2.9) node[anchor=north west] {$\pmb{x}$};
\draw (-2.42,2.84) node[anchor=north west] {$\pmb{x}$};
\draw (-1.94,1.66) node[anchor=north west] {$\mathcal{K}$};
\draw (2.3,1.6) node[anchor=north west] {$\widehat{\mathcal{H}}$};
\draw (2.1,2.9) node[anchor=north west] {$\mathcal{H}$};
\begin{scriptsize}
\draw [fill=black] (-1.76,2.54) circle (1.5pt);
\draw [fill=black] (4.16,2.66) circle (1.5pt);
\end{scriptsize}
\end{tikzpicture}
\caption{The graphs of Lemma \ref{splitlem}. }
\end{figure}
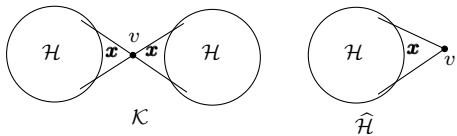

\begin{proof}
Suppose that $\pmb{\alpha}, \pmb{\beta},\pmb{\gamma}$ are  $\mu$-eigenvectors of $L_{\mathcal{H}}^{\rho}+D(\pmb{x})$, $L_{\mathcal{K}}^{\rho}$, and $L_{\widehat{\mathcal{H}}}^{\rho}$, respectively. \\
(i) $\Rightarrow$ (ii): For the index $I$ of $L_{\mathcal{K}}^{\rho}$, we have
{\footnotesize \begin{align}\label{equal}
(\mu\mathbb{I}-L_{\mathcal{H}}^{\rho}-D(\pmb{x}))\pmb{\beta}_{| I}=-
\rho \left(\begin{array}{c|c}
 \pmb{x}  & 0
\end{array}\right)
 \left(\begin{array}{c}
 \pmb{\beta}(v)  \\\cline{1-1}
 \pmb{\beta}_{| I'}
\end{array}\right).
\end{align}}
Multiplying relation (\ref{equal}) by $\pmb{\alpha}^{T}$ from the left, we obtain
\begin{center}
\footnotesize
$0=\pmb{\alpha}^{T}(\mu\mathbb{I}-L_{\mathcal{H}}^{\rho}-D(\pmb{x}))\pmb{\beta}_{| I}=
-\rho \pmb{\alpha}^{T}\pmb{x}\pmb{\beta}(v).$
\end{center}
Hence, $\pmb{\beta}(v)=0$ and $\pmb{\beta}_{| I}=a_1 \pmb{\alpha}$ and similarly $\pmb{\beta}_{| I'}=a_2 \pmb{\alpha}$, for some $a_1,a_2 \in\mathbb{R}$. 
For the index $v$ of $L_{\mathcal{K}}^{\rho}$, we have
{\footnotesize \begin{align*}
(\mu -a+\rho a-2D(\pmb{x}^{T}))\pmb{\beta}(v)=-
\rho (\pmb{x}^{T}\pmb{\beta}_{| I}+\pmb{x}^{T}\pmb{\beta}_{| I'})=-\rho (a_1+a_2)\pmb{x}^{T}\pmb{\alpha}.
\end{align*}}
Hence $a_2=-a_1$ and the proof is done.\\
(ii) $\Rightarrow$ (i): It follows by the relations above in a similar manner.\\
Now, we show that $m_{L_{\widehat{\mathcal{H}}}^{\rho}}(\mu)=0$. We have 
{\footnotesize \begin{align}\label{equal2}
(\mu\mathbb{I}-L_{\mathcal{H}}^{\rho}-D(\pmb{x}))\pmb{\gamma}_{| I}=-\rho \pmb{x}\pmb{\gamma}(v).
\end{align}}
Multiplying relation (\ref{equal2}) by $\pmb{\alpha}^{T}$ from the left, we obtain $\pmb{\gamma}(v)=0$ and $\pmb{\gamma}_{| I}=b\pmb{\alpha}$, for a $b \in\mathbb{R}$. 
For the index $v$ of $L_{\widehat{\mathcal{H}}}^{\rho}$, we have
{\footnotesize \begin{align*}
(\mu -a+\rho a-D(\pmb{x}^{T}))\pmb{\gamma}(v)=-
\rho \pmb{x}^{T}\pmb{\gamma}_{| I}=-\rho b \pmb{x}^{T}\pmb{\alpha}.
\end{align*}}
Hence $b=0$ and $m_{L_{\widehat{\mathcal{H}}}^{\rho}}(\mu)=0$.
\end{proof}

The following corollary is a straightforward consequence of Theorem \ref{spl} and Lemma \ref{splitlem}.
\begin{corollary}\label{splap}
Let  $\mu\in\mathbb{R},\rho\in\mathbb{R}-\{0\}$ and $H,L$  be two disjoint graphs and $u\in V(H)$ and $v\in V(L)$. Suppose that $H_{1},\ldots,H_{t}$ are $t$ copies of $H$ and $E, K, G$ are  graphs as shown below (see Figure \ref{splapfig}). If $m_{L_{K}^{\rho}}(\mu)=1$ and $\pmb{\beta}(v')=0$, for a  $\mu$-eigenvector $\pmb{\beta}$ of $L_{K}^{\rho}$,
then $m_{L_{G}^{\rho}}(\mu)=m_{L_{E}^{\rho}}(\mu)+t-1$.
\end{corollary}
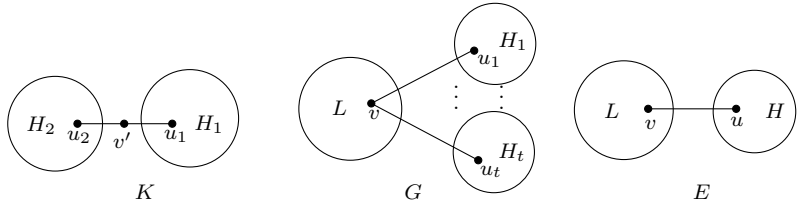
\begin{figure}[H]
\footnotesize
\centering
\begin{tikzpicture}[scale=.9,line cap=round,line join=round,>=triangle 45,x=1.0cm,y=1.0cm]
\draw(-5.86,2.18) circle (0.7cm);
\draw(-3.87,2.26) circle (0.72cm);
\draw (-5.53,2.18)-- (-4.84,2.18);
\draw (-4.13,2.18)-- (-4.84,2.18);
\draw (-4.37,2.21) node[anchor=north west] {$u_1$};
\draw (-5.8,2.19) node[anchor=north west] {$u_2$};
\draw (-5.12,2.18) node[anchor=north west] {$v'$};
\draw (-6.4,2.40) node[anchor=north west] {$H_2$};
\draw (-3.92,2.42) node[anchor=north west] {$H_1$};
\draw(-1.5,2.4) circle (0.76cm);
\draw(0.64,3.36) circle (0.60cm);
\draw(0.62,1.76) circle (0.6cm);
\draw (-1.19,2.48)-- (0.3,3.26);
\draw (-1.19,2.48)-- (0.39,1.64);
\draw (-1.89,2.64) node[anchor=north west] {$L$};
\draw (0.57,3.65) node[anchor=north west] {$H_1$};
\draw (0.57,3.1) node[anchor=north west] {$\vdots$};
\draw (0.57,2) node[anchor=north west] {$H_t$};
\draw (0.26,3.31) node[anchor=north west] {$u_1$};
\draw (-.1,3.1) node[anchor=north west] {$\vdots$};
\draw (0.3,1.68) node[anchor=north west] {$u_t$};
\draw (-1.35,2.5) node[anchor=north west] {$v$};
\draw(2.54,2.38) circle (0.73cm);
\draw (2.73,2.36) node[anchor=north west] {$v$};
\draw (2.13,2.6) node[anchor=north west] {$L$};
\draw(4.47,2.36) circle (0.61cm);
\draw (4.,2.4) node[anchor=north west] {$u$};
\draw (4.51,2.6) node[anchor=north west] {$H$};
\draw (2.90,2.4)-- (4.2,2.4);
\draw (-4.8,1.4) node[anchor=north west] {$K$};
\draw (-0.82,1.4) node[anchor=north west] {$G$};
\draw (3.44,1.4) node[anchor=north west] {$E$};
\begin{scriptsize}
\draw [fill=black] (-4.84,2.18) circle (1.5pt);
\draw [fill=black] (-4.13,2.18) circle (1.5pt);
\draw [fill=black] (-5.53,2.18) circle (1.5pt);
\draw [fill=black] (0.33,3.26) circle (1.5pt);
\draw [fill=black] (0.39,1.64) circle (1.5pt);
\draw [fill=black] (-1.19,2.48) circle (1.5pt);
\draw [fill=black] (2.9,2.4) circle (1.5pt);
\draw [fill=black] (4.2,2.4) circle (1.5pt);
\end{scriptsize}
\end{tikzpicture}
\caption{The graphs of Corollary \ref{splap}. }
\label{splapfig}
\end{figure}

Now, we give the second proof for Theorem \ref{pkqk}.\\

{\noindent \textbf{Second proof of Theorem \ref{pkqk}:} For brevity, we set $q_{k}(G)=q,\,\, p_{k}(G)=p$. With the notations of Corollary \ref{splap}, we put $\mu=4\cos^{2}(\frac{2t\pi}{2(2k+1)})$, $\rho=\pm 1, H=P_{k}$, then by using the splitting method of Corollary \ref{splap}, for $q$ vertices of $G$, $q$ times, we have  $m_{L_{G}^{\rho}}(4\cos^{2}(\frac{t\pi}{2k+1}))=m_{L_{E}^{\rho}}(4\cos^{2}(\frac{t\pi}{2k+1}))+p-q\geq p-q$, for $t\in [k]$.}\\

\begin{example}\label{mu1}\textbf{$\mu=4\cos^{2}(\frac{t\pi}{2k+1}):$}
Suppose that $k\in\mathbb{N}$ and $G,\,E$ are the  graphs as shown below (see Figure \ref{mu1fig}). Then $m_{L(G)}(\mu)=m_{L(E)}(\mu)$ and $m_{Q(G)}(\mu)=m_{Q(E)}(\mu)$.
\end{example}
\begin{figure}[H]
\footnotesize
\centering
\begin{tikzpicture}[scale=.6,line cap=round,line join=round,>=triangle 45,x=1.0cm,y=1.0cm]
\draw(-3.26,2.56) circle (1cm);
\draw (-1.76,2.54)-- (-3.02,3.2);
\draw (-1.76,2.54)-- (-2.98,2.02);
\draw (-2.1,3.22) node[anchor=north west] {$v$};
\draw(2.04,2.52) circle (1.cm);
\draw (3.9,3.14)-- (2.46,3.14);
\draw (3.88,1.98)-- (2.5,1.98);
\draw (3.9,3.14)-- (5.04,3.14);
\draw (3.88,1.98)-- (5.04,1.98);
\draw (-1.76,2.54)-- (-1.22,2.54);
\draw (3.54,3.78) node[anchor=north west] {$v_1$};
\draw (5.,3.4) node[anchor=north west] {$\vdots$};
\draw (3.55,1.98) node[anchor=north west] {$v_r$};
\draw (-3.92,3.54) node[anchor=north west] {$w_1$};
\draw (-3.25,3.4) node[anchor=north west] {$\vdots$};
\draw (-3.92,2.28) node[anchor=north west] {$w_r$};
\draw (4.54,3.78) node[anchor=north west] {$u_{11}$};
\draw (4.1,3.4) node[anchor=north west] {$\vdots$};
\draw (4.54,1.98) node[anchor=north west] {$u_{r1}$};
\draw (1.5,3.44) node[anchor=north west] {$w_1$};
\draw (2.2,3.4) node[anchor=north west] {$\vdots$};
\draw (1.5,2.24) node[anchor=north west] {$w_r$};
\draw (-3.94,3.) node[anchor=north west] {$L$};
\draw (1.46,3.) node[anchor=north west] {$L$};
\draw (-2.54,1.46) node[anchor=north west] {$G$};
\draw (2.94,1.46) node[anchor=north west] {$E$};
\draw (-0.62,2.54)-- (-0.06,2.56);
\draw (5.74,3.16)-- (6.28,3.18);
\draw (5.74,2.02)-- (6.28,2.02);
\draw (-1.6,2.52) node[anchor=north west] {$u_1$};
\draw (-1.3,2.81) node[anchor=north west] {$\cdots$};
\draw (-0.5,2.52) node[anchor=north west] {$u_k$};
\draw (5.94,3.78) node[anchor=north west] {$u_{1k}$};
\draw (4.94,3.42) node[anchor=north west] {$\cdots$};
\draw (4.94,2.27) node[anchor=north west] {$\cdots$};
\draw (5.94,1.98) node[anchor=north west] {$u_{rk}$};
\begin{scriptsize}
\draw [fill=black] (-1.22,2.54) circle (1.5pt);
\draw [fill=black] (-1.76,2.54) circle (1.5pt);
\draw [fill=black] (-3.02,3.2) circle (1.5pt);
\draw [fill=black] (-2.98,2.02) circle (1.5pt);
\draw [fill=black] (3.9,3.14) circle (1.5pt);
\draw [fill=black] (3.88,1.98) circle (1.5pt);
\draw [fill=black] (2.46,3.14) circle (1.5pt);
\draw [fill=black] (2.5,1.98) circle (1.5pt);
\draw [fill=black] (5.04,3.14) circle (1.5pt);
\draw [fill=black] (5.04,1.98) circle (1.5pt);
\draw [fill=black] (-0.62,2.54) circle (1.5pt);
\draw [fill=black] (-0.06,2.56) circle (1.5pt);
\draw [fill=black] (5.74,3.16) circle (1.5pt);
\draw [fill=black] (6.28,3.18) circle (1.5pt);
\draw [fill=black] (5.74,2.02) circle (1.5pt);
\draw [fill=black] (6.28,2.02) circle (1.5pt);
\draw (-11,3.) node[anchor=north west] {$\mu=4\cos^{2}(\frac{t\pi}{2k+1}):$};
\draw (-11,2.2) node[anchor=north west] {$t\in[k].$};
\end{scriptsize}
\end{tikzpicture}
\caption{The graphs of Example \ref{mu1}. }
\label{mu1fig}
\end{figure}
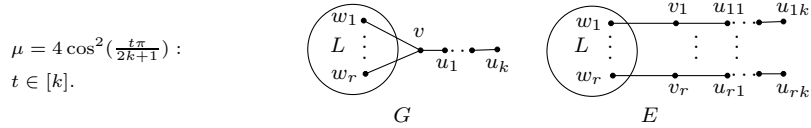

\section{Proofs of the Main Theorems}

For an $n\times m$ matrix $M$ and $I\subseteq [n], J\subseteq [m]$, let $M[I|J]$ denote the submatrix of $M$ formed by rows with index in $I$ and columns with index in $J$.

\begin{proof}[\textbf{Proof of Theorem \ref{symmetricII}}]
Suppose that $m_{L_{\mathcal{H}}^{\rho}}(\mu)=k$ and $\pmb{\alpha}^{1},\ldots,\pmb{\alpha}^{k}$ are the eigenvectors of $L_{\mathcal{H}}^{\rho}$ corresponding to $I_{1}\dot{\cup} I_{2}$ by Theorem \ref{starcorr}. We set
\begin{center}
\footnotesize
$E=\left(\begin{array}{c|c|c}
 \pmb{\alpha}^{1} & \cdots & \pmb{\alpha}^{k}
\end{array}\right)^{T}=
\left(\begin{array}{c|c}
 \mathbb{I}_{k} & \ast
\end{array}\right)$
\end{center}
and extend $\pmb{\alpha}^{i}$ to {\footnotesize $\widehat{\pmb{\alpha}^{i}}=\begin{blockarray}{cc}
\begin{block}{c(c)}
  I& \pmb{\alpha}^{i}\\\cline{1-2}
  J&\pmb{0}  \\
\end{block}
\end{blockarray}$,} for $i\in [k]$.
It is easy to check that $\widehat{\pmb{\alpha}^{1}},\ldots,\widehat{\pmb{\alpha}^{k}}$ are $k$ $\mu$-eigenvectors of $L_{\mathcal{E}}^{\rho}$.

Suppose that $\pmb{\beta}$ is a $\mu$-eigenvector of $L_{\mathcal{G}}^{\rho}$. We show that $\pmb{\beta}_{|I_{1j}}=\rho\pmb{\beta}(v_{j})\pmb{j}$, $j\in [J_{1}]$. We have
{\footnotesize \begin{align*}
(\mu\mathbb{I}-{L_{\mathcal{G}}^{\rho}})_{|I}\pmb{\beta}_{| I}=-
\rho \left(\begin{array}{c|c}
 X &   \\
 0 &\mathcal{A}\\
 0 &
\end{array}\right)
\pmb{\beta}_{| J}
\end{align*}
\begin{equation}\label{eqlam111}
(\mu\mathbb{I}-L_{\mathcal{H}}^{\rho}-D\left(\begin{array}{c|c}
 X  &   \\
 0 &\mathcal{A}\\
 0 &
\end{array}\right))\pmb{\beta}_{| I}=-
\rho \left(\begin{array}{c|c}
 X  &   \\
 0 &\mathcal{A}\\
 0 &
\end{array}\right)
\pmb{\beta}_{| J}.
\end{equation}}
Multiplying relation (\ref{eqlam111}) by $E$ from the left, we obtain
\begin{equation*}
\resizebox{ 0.7\textwidth}{!} 
{$
\pmb{0}=E(\mu\mathbb{I}-L_{\mathcal{H}}^{\rho})\pmb{\beta}_{| I}=
E\left(
\begin{array}{c}
    D(X)\pmb{\beta}_{| I_{1}}  \\ \cline{1-1}
    \pmb{0}   \\
   \pmb{0}
 \end{array}\right)-
\rho E\left(
\begin{array}{c}
  \pmb{\beta}_{|J_{1}}  X\\\cline{1-1}
   \pmb{0}\\
   \pmb{0}
\end{array}\right)-
\rho E\mathcal{A}\pmb{\beta}_{| J_{2}}=
\left(
\begin{array}{c}
    D(X)\pmb{\beta}_{| I_{1}}-\rho  \pmb{\beta}_{|J_{1}}  X  \\ \cline{1-1}
    \pmb{0}   \\
   \pmb{0}
 \end{array}\right).$}
\end{equation*}
Since, $\pmb{x}_j$ is nowhere-zero, hence, we have 
\begin{equation}\label{rel}
\pmb{\beta}_{| I_{1j}}=\rho \pmb{\beta}(v_{j})\pmb{j},\, j\in [|J_{1}|].
\end{equation}
Thus $\pmb{\beta}$ is a $\mu$-eigenvector of $L_{\mathcal{E}}^{\rho}$ by  Lemma \ref{edgeprin2}.

Now, we show that $m_{L_{\mathcal{E}}^{\rho}}(\mu)\geq m_{L_{\mathcal{G}}^{\rho}}(\mu)+|I_{1}|$.
Assume that $m_{L_{\mathcal{G}}^{\rho}}(\mu)=r$ and $\pmb{\beta}^{1},\ldots,\pmb{\beta}^{r}$ are independent eigenvectors of $L_{\mathcal{G}}^{\rho}$. We show that $\pmb{\beta}^{1},\ldots,\pmb{\beta}^{r}$, $\widehat{\pmb{\alpha}^{1}},\ldots,\widehat{\pmb{\alpha}^{|I_{1}|}}$ are independent. Suppose that for some $c_{1},\ldots,c_{r},d_{1},\ldots,d_{|I_{1}|}\in\mathbb{R}$,
\begin{center}
$\sum_{i=1}^{r}c_{i}\pmb{\beta}^{i}=\sum_{i=1}^{|I_{1}|}d_{i}\widehat{\pmb{\alpha}^{i}}.$
\end{center}
Hence $\sum_{i=1}^{r}c_{i}\pmb{\beta}^{i}=\sum_{i=1}^{|I_{1}|}d_{i}\widehat{\pmb{\alpha}^{i}}=(d_{1},\ldots,d_{|I_{1}|},\ast,\ldots,\ast)^{T}\,.$

Thus $(\mu\mathbb{I}-L_{\mathcal{G}}^{\rho})(\sum_{i=1}^{|I_{1}|}d_{i}\widehat{\pmb{\alpha}^{i}})=\pmb{0}$. From relation (\ref{rel}),
$(d_{1},\ldots,d_{|I_{1}|})^{T}=\pmb{0}$ and so $c_{1}=\cdots=c_{r}=0$ and $m_{L_{\mathcal{E}}^{\rho}}(\mu)\geq m_{L_{\mathcal{G}}^{\rho}}(\mu)+|I_{1}|$.

Next, we show that $m_{L_{\mathcal{E}}^{\rho}}(\mu)\leq m_{L_{\mathcal{G}}^{\rho}}(\mu)+|I_{1}|$.

 $\widehat{\pmb{\alpha}^{1}},\ldots,\widehat{\pmb{\alpha}^{|I_{1}|}}$ are $|I_{1}|$ $\mu$-eigenvectors of $L_{\mathcal{E}}^{\rho}$. Suppose that $m_{L_{\mathcal{E}}^{\rho}}(\mu)=s+|I_{1}|$ and $\widehat{\pmb{\alpha}^{1}},\ldots,\widehat{\pmb{\alpha}^{|I_{1}|}}$, $\pmb{\gamma}^{1},\ldots,\pmb{\gamma}^{s}$ are independent $\mu$-eigenvectors of $L_{\mathcal{E}}^{\rho}$.
For $i\in [s]$, we define $\widehat{\pmb{\gamma}^{i}}$ as below,
\begin{equation*}
\resizebox{ 0.4\textwidth}{!} 
{$\widehat{\pmb{\gamma}^{i}}=\pmb{\gamma}^{i}+
\begin{blockarray}{cc}
\begin{block}{c(c)}
I&\sum_{j\in [|J_{1}|]} E^{T}[I|I_{1j}](\rho\pmb{\gamma}^{i}(v_{j})\pmb{j}_{|I_{1}|}-\pmb{\gamma}^{i}_{|I_{1j}}) \\\cline{1-2}
J_{1}& \pmb{0}\\\cline{1-2}
J_{2}& \pmb{0}\\
\end{block}
\end{blockarray}\,.$}
\end{equation*}
We show that $\widehat{\pmb{\gamma}^{1}},\ldots,\widehat{\pmb{\gamma}^{s}}$ are $s$ independent $\mu$-eigenvectors of $L_{\mathcal{G}}^{\rho}$. We have

\begin{equation*}
\resizebox{ 0.8\textwidth}{!} 
{$\begin{aligned}
&L_{\mathcal{G}}^{\rho}\widehat{\pmb{\gamma}^{i}}=(L_{\mathcal{E}}^{\rho}+
\begin{blockarray}{cccc}
I_{1} & I\setminus I_{1}&\{v\}&J_{1}\\
\begin{block}{(cccc)}
 D(X)  & 0 & -\rho X & 0\\
 0  & 0 & 0 & 0\\
  -\rho X^{T} & 0 &D(X^{T})  & 0 \\
 0 & 0 & 0 & 0 \\
\end{block}
\end{blockarray}\,)(\pmb{\gamma}^{i}+
\begin{pmatrix}
\sum_{j\in [|J_{1}|]} E^{T}[I|I_{1j}](\rho\pmb{\gamma}^{i}(v_{j})\pmb{j}_{|I_{1j}|}-\pmb{\gamma}^{i}_{|I_{1j}}) \\
 \pmb{0}\\
 \pmb{0}
\end{pmatrix})\\
&=\mu\pmb{\gamma}^{i}+
\begin{pmatrix}
\mu \sum_{j\in [|J_{1}|]} E^{T}[I|I_{1j}](\rho\pmb{\gamma}^{i}(v_{j})\pmb{j}_{|I_{1j}|}-\pmb{\gamma}^{i}_{|I_{1j}}) \\
 \pmb{0}\\
 \pmb{0}
\end{pmatrix}+
\begin{pmatrix}
D(\pmb{x})\pmb{\gamma}^{i}_{|I_{1}}-\rho\pmb{\gamma}^{i}(v)\pmb{x}\\\
  \pmb{0}\\
  -\rho \pmb{x}^{T}\pmb{\gamma}^{i}_{|I_{1}}+D(\pmb{x}^{T})\pmb{\gamma}^{i}(v) \\
  \pmb{0}
\end{pmatrix}+
\begin{pmatrix}
\rho\pmb{\gamma}^{i}(v)\pmb{x}-D(\pmb{x})\pmb{\gamma}^{i}_{|I_{1}}\\
  \pmb{0}\\
 -{\rho}^{2}\pmb{\gamma}^{i}(v)\pmb{x}^{T}\pmb{j}+ \rho \pmb{x}^{T}\pmb{\gamma}^{i}_{|I_{1}} \\
  \pmb{0}
\end{pmatrix}\\
&=\mu \widehat{\pmb{\gamma}^{i}}.
\end{aligned}$}
\end{equation*}
Now, suppose that
\begin{center}
\footnotesize
$\pmb{0}=\sum_{i=1}^{s}c_{i}\widehat{\pmb{\gamma}^{i}}=
\sum_{i=1}^{s}c_{i}\pmb{\gamma}^{i}+
\begin{pmatrix}
\sum_{j\in [|J_{1}|]} E^{T}[I|I_{1j}](\rho\pmb{\gamma}^{i}(v_{j})\pmb{j}_{|I_{1}|}-\pmb{\gamma}^{i}_{|I_{1j}}) \\
 \pmb{0}\\
 \pmb{0}
\end{pmatrix},$
\end{center}
for some $c_{1},\ldots,c_{s}\in\mathbb{R}$. So, $\sum_{i=1}^{s}c_{i}\pmb{\gamma}^{i}=\sum_{i\in [s],j\in [|I_{1}|]}d_{ij}\widehat{\pmb{\alpha}^{j}}$, for some real numbers $d_{ij}$.
Thus $c_{1}=\cdots=c_{s}=0$ and hence $m_{L_{\mathcal{E}}^{\rho}}(\mu)\leq m_{L_{\mathcal{G}}^{\rho}}(\mu)+|I_{1}|$.
This inequality implies that $m_{L_{\mathcal{E}}^{\rho}}(\mu)=m_{L_{\mathcal{G}}^{\rho}}(\mu)+|I_{1}|$.
\end{proof}

\begin{proof}[\textbf{Proof of Theorem \ref{switching}}]
Suppose that $m_{L_{\mathcal{H}}^{\rho}}(\mu)=k$ and $\pmb{\alpha}^{1},\ldots,\pmb{\alpha}^{k}$ are the eigenvectors of $L_{\mathcal{H}}^{\rho}$ corresponding to $I_{1}\dot{\cup} I_{2}$ by Theorem \ref{starcorr}. We set
{\footnotesize
$E=\left(\begin{array}{c|c|c}
 \pmb{\alpha}^{1} & \cdots & \pmb{\alpha}^{k}
\end{array}\right)^{T}=
\left(\begin{array}{c|c}
 \mathbb{I}_{k} & \ast
\end{array}\right)$}
and extend $\pmb{\alpha}^{i}$ to {\footnotesize $\widehat{\pmb{\alpha}^{i}}=\begin{blockarray}{cc}
\begin{block}{c(c)}
  I& \pmb{\alpha}^{i}\\\cline{1-2}
  J&\pmb{0}  \\
\end{block}
\end{blockarray}$,} for $i\in [k]$.
It is easy to check that $\widehat{\pmb{\alpha}^{1}},\ldots,\widehat{\pmb{\alpha}^{k}}$ are $k$ $\mu$-eigenvectors of $L_{\mathcal{E}}^{\rho}$.

Suppose that $\pmb{\beta}$ is a $\mu$-eigenvector of $L_{\mathcal{G}}^{\rho}$. We have
{\footnotesize \begin{align*}
(\mu\mathbb{I}-{L_{\mathcal{G}}^{\rho}})_{|I}\pmb{\beta}_{| I}=-
\rho \left(\begin{array}{c|c}
 X  &   \\
 0 &\mathcal{A}\\
 0 &
\end{array}\right)
\pmb{\beta}_{| J}
\end{align*}
\begin{equation}\label{eqlam12}
(\mu\mathbb{I}-L_{\mathcal{H}}^{\rho}-D\left(\begin{array}{c|c}
 X  &   \\
 0 &\mathcal{A}\\
 0 &
\end{array}\right))\pmb{\beta}_{| I}=
\left(\begin{array}{c|c}
  L_{S}^{\rho} & 0  \\ \cline{1-2}
 0 & 0
\end{array}\right)
\pmb{\beta}_{| I}-
\rho \left(\begin{array}{c|c}
 X  &   \\
 0 &\mathcal{A}\\
 0 &
\end{array}\right)
\pmb{\beta}_{| J}.
\end{equation}}
Multiplying relation (\ref{eqlam12}) by $E$ from the left, we obtain
\begin{equation*}
\resizebox{ 0.8\textwidth}{!} 
{$
\pmb{0}=E(\mu\mathbb{I}-L_{\mathcal{H}}^{\rho})\pmb{\beta}_{| I}=
E\left(
\begin{array}{c}
   (L_{S}^{\rho}+D(X))\pmb{\beta}_{| I_{1}}  \\ \cline{1-1}
    \pmb{0}   \\
   \pmb{0}
 \end{array}\right)-
\rho E\left(
\begin{array}{c}
   X\pmb{\beta}_{| J_{1}} \\\cline{1-1}
   \pmb{0}\\
   \pmb{0}
\end{array}\right)-
\rho E\mathcal{A}\pmb{\beta}_{| J_{2}}=
\left(
\begin{array}{c}
     (L_{S}^{\rho}+D(X))\pmb{\beta}_{| I_{1}}-\rho X\pmb{\beta}_{| J_{1}}\\ \cline{1-1}
    \pmb{0}   \\
   \pmb{0}
 \end{array}\right).$}
\end{equation*}
Hence, 
{\footnotesize \begin{equation}\label{rel2}
\pmb{\beta}_{| I_{1}}=\rho  (L_{S}^{\rho}+D(X))^{-1}X\pmb{\beta}_{| J_{1}}.
\end{equation}}
So, $\rho X^{T}\pmb{\beta}_{| I_{1}}=(L_{S'}^{\rho}+D(X^{T}))\pmb{\beta}_{| J_{1}}$ and it is easy to see that $\pmb{\beta}$ is a $\mu$-eigenvector of $L_{\mathcal{E}}^{\rho}$.

Now, we show that $m_{L_{\mathcal{E}}^{\rho}}(\mu)\geq m_{L_{\mathcal{G}}^{\rho}}(\mu)+|I_{1}|$.
Assume that $m_{L_{\mathcal{G}}^{\rho}}(\mu)=r$ and $\pmb{\beta}^{1},\ldots,\pmb{\beta}^{r}$ are independent eigenvectors of $L_{\mathcal{G}}^{\rho}$. We show that $\pmb{\beta}^{1},\ldots,\pmb{\beta}^{r}$, $\widehat{\pmb{\alpha}^{1}},\ldots,\widehat{\pmb{\alpha}^{|I_{1}|}}$ are independent. Suppose that for some $c_{1},\ldots,c_{r},d_{1},\ldots,d_{|I_{1}|}\in\mathbb{R}$,
\begin{center}
$\sum_{i=1}^{r}c_{i}\pmb{\beta}^{i}=\sum_{i=1}^{|I_{1}|}d_{i}\widehat{\pmb{\alpha}^{i}}=(d_{1},\ldots,d_{|I_{1}|},\ast,\ldots,\ast)^{T}.$
\end{center}

Thus $(\mu\mathbb{I}-L_{\mathcal{G}}^{\rho})(\sum_{i=1}^{|I_{1}|}d_{i}\widehat{\pmb{\alpha}^{i}})=\pmb{0}$. From relation (\ref{rel2}),
$(d_{1},\ldots,d_{|I_{1}|})^{T}=\pmb{0}$ and so $c_{1}=\cdots=c_{r}=0$. Hence, $m_{L_{\mathcal{E}}^{\rho}}(\mu)\geq m_{L_{\mathcal{G}}^{\rho}}(\mu)+|I_{1}|$.

Next, we show that $m_{L_{\mathcal{E}}^{\rho}}(\mu)\leq m_{L_{\mathcal{G}}^{\rho}}(\mu)+|I_{1}|$.

 $\widehat{\pmb{\alpha}^{1}},\ldots,\widehat{\pmb{\alpha}^{|I_{1}|}}$ are $|I_{1}|$ $\mu$-eigenvectors of $L_{\mathcal{E}}^{\rho}$. Suppose that $m_{L_{\mathcal{E}}^{\rho}}(\mu)=s+|I_{1}|$ and $\widehat{\pmb{\alpha}^{1}},\ldots,\widehat{\pmb{\alpha}^{|I_{1}|}}$, $\pmb{\gamma}^{1},\ldots,\pmb{\gamma}^{s}$ are independent $\mu$-eigenvectors of $L_{\mathcal{E}}^{\rho}$.
For $i\in [s]$, we define $\widehat{\pmb{\gamma}^{i}}$ as below,
{\footnotesize \begin{equation*}
\widehat{\pmb{\gamma}^{i}}=\pmb{\gamma}^{i}+
\begin{blockarray}{cc}
\begin{block}{c(c)}
I&E^{T}[I|I_{1}]( \rho (L_{S}^{\rho}+D(X))^{-1}X\pmb{\gamma}^{i}_{|J_{1}}-\pmb{\gamma}^{i}_{|I_{1}}) \\\cline{1-2}
J& \pmb{0}\\
\end{block}
\end{blockarray}\,.
\end{equation*}}
We show that $\widehat{\pmb{\gamma}^{1}},\ldots,\widehat{\pmb{\gamma}^{s}}$ are $s$ independent $\mu$-eigenvectors of $L_{\mathcal{G}}^{\rho}$. We have
\begin{equation*}
\resizebox{ 0.95\textwidth}{!} 
{$
\begin{aligned}
&L_{\mathcal{G}}^{\rho}\widehat{\pmb{\gamma}^{i}}=(L_{\mathcal{E}}^{\rho}+
\begin{blockarray}{cccc}
I_{1} & I\setminus I_{1}&J_{1}&J_{3}\\
\begin{block}{(cccc)}
 L_{S}^{\rho}+D(X) & 0 & -\rho X & 0\\
 0  & 0 & 0 & 0\\
 -\rho X^{T} & 0 &L_{S'}^{\rho}+D(X^{T})  & 0 \\
 0 & 0 & 0 & 0 \\
\end{block}
\end{blockarray}\,)(\pmb{\gamma}^{i}+
\begin{pmatrix}
E^{T}[I|I_{1}]( \rho (L_{S}^{\rho}+D(X))^{-1}X\pmb{\gamma}^{i}_{|J_{1}}-\pmb{\gamma}^{i}_{|I_{1}}) \\
 \pmb{0}\\
 \pmb{0}
\end{pmatrix})\\
&=\mu\pmb{\gamma}^{i}+
\begin{pmatrix}
\mu E^{T}[I|I_{1}]( \rho (L_{S}^{\rho}+D(X))^{-1}X\pmb{\gamma}^{i}_{|J_{1}}-\pmb{\gamma}^{i}_{|I_{1}}) \\
 \pmb{0}\\
 \pmb{0}
\end{pmatrix}+
\begin{pmatrix}
   (L_{S}^{\rho}+D(X))\pmb{\gamma}_{| I_{1}}-\rho X\pmb{\gamma}_{| J_{1}}\\
  \pmb{0}\\
  -\rho X^{T}\pmb{\gamma}_{| I_{1}}+(L_{S'}^{\rho}+D(X^{T}))\pmb{\gamma}_{| J_{1}} \\
  \pmb{0}
\end{pmatrix}+
\begin{pmatrix}
 \rho X\pmb{\gamma}_{| J_{1}}-(L_{S}^{\rho}+D(X))\pmb{\gamma}_{| I_{1}}\\
  \pmb{0}\\
  -{\rho}^{2}X^{T}(L_{S}^{\rho}+D(X))^{-1}X\pmb{\gamma}^{i}_{|J_{1}}+\rho X^{T}\pmb{\gamma}^{i}_{|I_{1}} \\
  \pmb{0}
\end{pmatrix}\\
&=\mu \widehat{\pmb{\gamma}^{i}}.
\end{aligned}$}
\end{equation*}
Now, suppose that
\begin{equation*}
\resizebox{0.64 \textwidth}{!} 
{$\pmb{0}=\sum_{i=1}^{s}c_{i}\widehat{\pmb{\gamma}^{i}}=
\sum_{i=1}^{s}c_{i}\pmb{\gamma}^{i}+
\begin{pmatrix}
E^{T}[I|I_{1}](\rho(L_{S}^{\rho}+D(X))^{-1}X\sum_{i=1}^{s}c_{i}\pmb{\gamma}^{i}_{|J_{1}}-\sum_{i=1}^{s}c_{i}\pmb{\gamma}^{i}_{|I_{1}}) \\
 \pmb{0}\\
 \pmb{0}
\end{pmatrix},$}
\end{equation*}
for some $c_{1},\ldots,c_{s}\in\mathbb{R}$. So,  $\sum_{i=1}^{s}c_{i}\pmb{\gamma}^{i}=\sum_{i\in [s],i\in [|I_{1}|]}d_{ij}\widehat{\pmb{\alpha}^{j}}$, for some real numbers $d_{ij}$.
Thus $c_{1}=\cdots=c_{s}=0$ and hence $m_{L_{\mathcal{E}}^{\rho}}(\mu)\leq m_{L_{\mathcal{G}}^{\rho}}(\mu)+|I_{1}|$.
This inequality implies that $m_{L_{\mathcal{E}}^{\rho}}(\mu)=m_{L_{\mathcal{G}}^{\rho}}(\mu)+|I_{1}|$.
Finally, to prove $m_{L_{\widehat{\mathcal{H}}}^{\rho}}(\mu)=m_{L_{\mathcal{H}}^{\rho}}(\mu)-|I_{1}|$, it suffices to put $X=0$ and $\mathcal{L}=0$.
\end{proof}


\begin{proof}[\textbf{Proof of Theorem \ref{cluster}}]

We extend $\pmb{\beta}^{i}$ to $\widehat{\pmb{\beta}^{i}}$ and $\pmb{\gamma}^{i}$ to $\widehat{\pmb{\gamma}^{i}}$, where
\begin{equation*}
\resizebox{0.4 \textwidth}{!}{ 
$\widehat{\pmb{\beta}^{i}}=\begin{blockarray}{cc}
\begin{block}{c(c)}
  I& \pmb{\beta}^{i}_{|I}\\
  I_{1}&\pmb{0}  \\
\vdots&\vdots  \\
  I_{r-1}&\pmb{0}  \\
  J&\pmb{\beta}^{i}_{|J}  \\
  J_{1}&\pmb{\beta}^{i}_{|J_{1}}  \\ 
\end{block}
\end{blockarray}\,\,,\,\, i\in [s],\,\,
\widehat{\pmb{\gamma}^{i}}=\begin{blockarray}{cc}
\begin{block}{c(c)}
  I& \pmb{\gamma}^{i}_{|I}\\
  I_{1}&\pmb{0}  \\
\vdots&\vdots  \\
 I_{i}& \pmb{\gamma}^{i}_{|I_{i}}\\
\vdots&\vdots  \\
  I_{r-1}&\pmb{0}  \\
  J&\pmb{\gamma}^{i}_{|J}  \\
  J_{1}&\pmb{0}  \\ 
\end{block}
\end{blockarray}\,\,,\,\,i\in [r-1].$
}\end{equation*}
It is easy to check that $\{\widehat{\pmb{\beta}^{i}}\}_{i=1}^{s}$ and $\{\widehat{\pmb{\gamma}^{i}}\}_{i=1}^{r-1}$ are  $\mu$-eigenvectors of $L_{\mathcal{G}}^{\rho}$. By the definitions, the independence of these eigenvectors is obvious.
\end{proof}


\begin{proof}[\textbf{Proof of Theorem \ref{spl}}]
Suppose that  $\pmb{\alpha}$ is the $\mu$-eigenvector of $L_{\mathcal{H}}^{\rho}+D(\pmb{x})$ such that $\pmb{x}^{T}\pmb{\alpha}=1$ and $\pmb{\beta}$ is a $\mu$-eigenvector of $L_{\mathcal{G}}^{\rho}$. We have
{\footnotesize \begin{align*}
(\mu\mathbb{I}-{L_{\mathcal{G}}^{\rho}})_{|I}\pmb{\beta}_{| I}=-
\rho \left(\begin{array}{c|c}
 \pmb{x}  & 0
\end{array}\right)
 \left(\begin{array}{c}
 \pmb{\beta}(v)  \\\cline{1-1}
 \pmb{\beta}_{| J}
\end{array}\right)
\end{align*}
\begin{equation}\label{eqlam121}
(\mu\mathbb{I}-L_{\mathcal{H}}^{\rho}-D\left(\begin{array}{c|c}
 \pmb{x}  & 0 
\end{array}\right))\pmb{\beta}_{| I}=
-\rho \left(\begin{array}{c|c}
 \pmb{x}  & 0  
\end{array}\right)
 \left(\begin{array}{c}
 \pmb{\beta}(v)  \\\cline{1-1}
 \pmb{\beta}_{| J}
\end{array}\right).
\end{equation}}
Multiplying relation (\ref{eqlam121}) by $\pmb{\alpha}^{T}$ from the left, we obtain
\begin{center}
\footnotesize
$0=\pmb{\alpha}^{T}(\mu\mathbb{I}-L_{\mathcal{H}}^{\rho}-D(\pmb{x}))\pmb{\beta}_{| I}=
-\rho \pmb{\alpha}^{T}\pmb{x}\pmb{\beta}(v).$
\end{center}
Hence, $\pmb{\beta}(v)=0$ and $\pmb{\beta}_{| I}=a\pmb{\alpha}$ for an $a\in\mathbb{R}$ .
For the index $v$ of $L_{\mathcal{G}}^{\rho}$, we have
{\footnotesize \begin{align*}
(\mu -a+\rho a-D(\pmb{x}^{T})-D(\pmb{y}^{T}))\pmb{\beta}(v)=-
\rho (\pmb{x}^{T}\pmb{\beta}_{| I}+\pmb{y}^{T}\pmb{\beta}_{| J}).
\end{align*}}
Hence, 
{\footnotesize \begin{equation}\label{rel222}
\pmb{x}^{T}\pmb{\beta}_{| I}=-\pmb{y}^{T}\pmb{\beta}_{| J},\quad a=-\frac{\pmb{y}^{T}\pmb{\beta}_{| J}}{\pmb{x}^{T}\pmb{\alpha}}=-\pmb{y}^{T}\pmb{\beta}_{| J}.
\end{equation}}
By similar method, if $\pmb{\gamma}$ is a $\mu$-eigenvector of $L_{\mathcal{E}}^{\rho}$, then
\begin{equation}\label{gam}
\pmb{\gamma}(v_{j})=0,\quad \pmb{\gamma}_{| I_{j}}=a_{j}\pmb{\alpha},\text{ and } a_{j}=\pmb{x}^{T}\pmb{\gamma}_{| I_{j}}=-y_{j}\pmb{\gamma}(j),\text{ for } j\in [|J|]. 
\end{equation}
Now, we show that $m_{L_{\mathcal{E}}^{\rho}}(\mu)\geq m_{L_{\mathcal{G}}^{\rho}}(\mu)$.
Assume that $m_{L_{\mathcal{G}}^{\rho}}(\mu)=r$ and $\pmb{\beta}^{1},\ldots,\pmb{\beta}^{r}$ are independent $\mu$-eigenvectors of $L_{\mathcal{G}}^{\rho}$. Put
{\footnotesize \[
\widehat{\pmb{\beta}^{i}}=\begin{blockarray}{cc}
\begin{block}{c(c)}
I_{1} & -y_{1}\pmb{\beta}^{i}(1)\pmb{\alpha} \\ \cline{1-2}
v_{1} & 0\\\cline{1-2}
\vdots&\vdots\\\cline{1-2}
I_{|J|} & -y_{|J|}\pmb{\beta}^{i}(|J|)\pmb{\alpha}  \\ \cline{1-2}
v_{|J|} & 0\\\cline{1-2}
J& \pmb{\beta}^{i}_{|J}\\
\end{block}
\end{blockarray},\quad i\in [r]. 
\]}
It is easy to see that $\widehat{\pmb{\beta}^{i}}$ is a $\mu$-eigenvector of $L_{\mathcal{E}}^{\rho}$. We show that  $\widehat{\pmb{\beta}^{1}},\ldots,\widehat{\pmb{\beta}^{r}}$ are independent.
From the relation (\ref{rel222}), we can conclude that $\pmb{\beta}^{1},\ldots,\pmb{\beta}^{r}$ are independent, if and only if $\pmb{\beta}^{1}_{|J},\ldots,\pmb{\beta}^{r}_{|J}$ are independent. So, $\widehat{\pmb{\beta}^{1}},\ldots,\widehat{\pmb{\beta}^{s}}$ are independent and hence, $m_{L_{\mathcal{E}}^{\rho}}(\mu)\geq m_{L_{\mathcal{G}}^{\rho}}(\mu)$.

Next, we show that $m_{L_{\mathcal{E}}^{\rho}}(\mu)\leq m_{L_{\mathcal{G}}^{\rho}}(\mu)$.
Suppose that $m_{L_{\mathcal{E}}^{\rho}}(\mu)=s$ and $\pmb{\gamma}^{1},\ldots,\pmb{\gamma}^{s}$ are independent $\mu$-eigenvectors of $L_{\mathcal{E}}^{\rho}$.
For $i\in [s]$, put
{\footnotesize \[
\widehat{\pmb{\gamma}^{i}}=\begin{blockarray}{cc}
\begin{block}{c(c)}
I& -(\pmb{y}^{T}\pmb{\gamma}^{i}_{|J})\pmb{\alpha} \\ \cline{1-2}
v& 0\\\cline{1-2}
J& \pmb{\gamma}^{i}_{|J}\\
\end{block}
\end{blockarray}. 
\]}
It is easy to see that $\widehat{\pmb{\gamma}^{i}}$ is a $\mu$-eigenvector of $L_{\mathcal{G}}^{\rho}$.
From the relation (\ref{gam}), $\pmb{\gamma}^{1},\ldots,\pmb{\gamma}^{s}$ are independent, if and only if $\pmb{\gamma}^{1}_{|J},\ldots,\pmb{\gamma}^{s}_{|J}$ are independent. So, $\widehat{\pmb{\gamma}^{1}},\ldots,\widehat{\pmb{\gamma}^{s}}$ are independent and hence $m_{L_{\mathcal{E}}^{\rho}}(\mu)\leq m_{L_{\mathcal{G}}^{\rho}}(\mu)$.
This inequality implies that $m_{L_{\mathcal{E}}^{\rho}}(\mu)=m_{L_{\mathcal{G}}^{\rho}}(\mu)$.
\end{proof}

\subsection*{Acknowledgements}
We would like to thank the unknown referee for his/her valuable comments. The authors are indebted to the School of Mathematics, Institute for Research in Fundamental
Sciences (IPM), Tehran, Iran for the support. The research of the second author were in part supported by grant from IPM (No. 94050116).



\begin{thebibliography}{30}



\bibitem{BK} A. Bahmani, D. Kiani.  \newblock On the multiplicity of the adjacency eigenvalues of graphs. \newblock {\em Linear Algebra Appl.}, 477:1--20, 2015.

\bibitem{CRS} D. Cvetkovi\'{c}, P. Rowlinson, S. Simi\'{c}.  \newblock {\em Eigenspaces of Graphs.} Cambridge University Press, Cambridge, 1997.

\bibitem{Fa} I. Faria. \newblock Permanental roots and the star degree of a graph.  \newblock {\em Linear Algebra Appl.}, 64:255-265, 1985.

\bibitem{Gh} E. Ghorbani. \newblock Proof of a conjecture on ‘plateaux’ phenomenon of graph Laplacian eigenvalues. \newblock 2015. \arxiv{1510.05117}.

\bibitem{GM} R. Grone, R. Merris. \newblock  The Laplacian  spectrum of a graph II.  \newblock {\em SIAM Journal on discrete Mathematics,} 7:221--229, 1994. 

\bibitem{GMS} R. Grone, R. Merris, V. S. Sunder. \newblock  The Laplacian spectrum of a graph.  \newblock {\em SIAM J. Matrix Anal. Appl.}, 11:218--238, 1990.


\bibitem{Me}  R. Merris.  \newblock Laplacian  Graph Eigenvectors.  \newblock {\em Linear Algebra Appl.}, 278:221--236, 1998. 


\bibitem{Ne} M. W. Newman. \newblock  The Laplacian spectrum of graphs.  \newblock {\em Master Diss..} University of Manitoba, Canada, 2000.


\bibitem{SW} N. Saito, E. Woei.  \newblock Tree simplification and the ‘plateaux’ phenomenon of graph Laplacian eigenvalues.  \newblock {\em Linear Algebra Appl.}, 481:263--279, 2015.

\bibitem{Sp} D. Spielman. \newblock  Spectral Graph Theory, The Laplacian.  \newblock {\em University Lecture Notes.}, 2009, available at: http://www.cs.yale.edu/homes/spielman/561/2009/lect02-09.pdf

\end{thebibliography}
\end{document}